\documentclass[11pt]{article}
\usepackage{amssymb, amsthm, amsmath, amsfonts}
\usepackage[margin=1in]{geometry}
\usepackage{wasysym}
\usepackage{mathrsfs}
\usepackage{hyperref}
\usepackage{graphicx}
\usepackage{lineno}
\usepackage[backgroundcolor=gray!30,linecolor=black]{todonotes}
\usepackage{listings}
\usepackage[T1]{fontenc}
\usepackage{cancel, enumerate}
\usepackage{rotating, environ}
\usepackage{caption}
\usepackage{subcaption}
\usepackage[inline]{enumitem}
\usepackage{dirtree}
\usepackage{xcolor}
\usepackage{commath}
\usepackage{arydshln}
\usepackage{pgfplots}
\usepackage{tikz}
\usepackage{float}
\usepackage{mathtools}
\usepackage{amstext}	
\usepackage{amscd}
\usepackage{tikz-cd}
\usepackage{amsbsy}
\usepackage{esdiff}
\usepackage{esint}
\usepackage{xurl}
\usepackage{wrapfig}

\makeatletter
\newcommand*\bdot{\mathpalette\bdot@{.65}}
\newcommand*\bdot@[2]{\mathbin{\vcenter{\hbox{\scalebox{#2}{$\m@th#1\bullet$}}}}}
\makeatother

\def\R{\mathbb R}

\def\N{\mathbb N}

\def\C{\mathbb C}
\def\grad{\nabla}
\def\G{\mathscr G}

\def\eps{\varepsilon}
\def\bx{\mathbf{x}}
\def\by{\mathbf{y}}
\def\supp{\mathop{\mathrm{supp}}}

\renewcommand{\t}{\text}
\def\defn{\coloneqq}
\newtheoremstyle{mine}{1.5\baselineskip}{1.5\baselineskip}{}{}{\bfseries}{:}{.5em}{}
\theoremstyle{mine}
\newtheorem{definition}{Definition}[section]
\newtheorem{Prop}[definition]{Proposition}
\newtheorem{thm}[definition]{Theorem}
\newtheorem{coro}[definition]{Corollary}
\newtheorem{lemma}[definition]{Lemma}

\newtheorem{examplex}[definition]{Example}
\newenvironment{ex}
{\pushQED{\qed}\examplex}
{\popQED\endexamplex}

\newtheorem{remark}[definition]{Remark}

\definecolor{codegreen}{rgb}{0,0.6,0}
\definecolor{codegray}{rgb}{0.5,0.5,0.5}
\definecolor{codepurple}{rgb}{0.58,0,0.82}
\definecolor{backcolour}{rgb}{0.95,0.95,0.92}
\lstdefinestyle{mystyle}{
	backgroundcolor=\color{backcolour},   
	commentstyle=\color{codegreen},
	keywordstyle=\color{magenta},
	numberstyle=\tiny\color{codegray},
	stringstyle=\color{codepurple},
	basicstyle=\ttfamily\footnotesize,
	breakatwhitespace=false,         
	breaklines=true,                 
	captionpos=b,                    
	keepspaces=true,                 
	numbers=left,                    
	numbersep=5pt,                  
	showspaces=false,                
	showstringspaces=false,
	showtabs=false,                  
	tabsize=2
}
\usetikzlibrary{arrows.meta}
\pgfplotsset{compat=1.17}
\def\D{\mathcal{D}}

\renewcommand{\norm}[1]{\left\lVert#1\right\rVert}
\hypersetup{
	colorlinks=true,
	linkcolor=blue,
	citecolor=[rgb]{.0,.527,.243}}

\bibliography{./References}

\DeclareFontFamily{U}{mathx}{}
\DeclareFontShape{U}{mathx}{m}{n}{<-> mathx10}{}
\DeclareSymbolFont{mathx}{U}{mathx}{m}{n}
\DeclareMathAccent{\widehat}{0}{mathx}{"70}
\DeclareMathAccent{\widecheck}{0}{mathx}{"71}

\newcommand{\dom}{\Omega}

\newcommand{\Ds}{\overline{\D}_{\text{s}}}
\newcommand{\Da}{\overline{\D}_{\text{a}}}

\newcommand{\lp}{\left(}
\newcommand{\rp}{\right)}
\newcommand{\ls}{\left[}
\newcommand{\rs}{\right]}

\newcommand{\qqquad}{\qquad\qquad}
\newcommand{\qqqquad}{\qqquad\qqquad}

\title{Nonlocal Differential Operators with Integrable, Nonsymmetric Kernels: Part I--Operator Theoretic Properties}
\author{Mikil Foss, Michael Pieper}

\begin{document}
	\maketitle
	\tableofcontents
	
	\begin{abstract}
		Recent decades have provided a host of examples and applications motivating the study of nonlocal differential operators. We discuss a class of such operators acting on bounded domains, focusing on those with integrable kernels having compact support. Notably, we make no explicit symmetry assumptions on the kernel and discuss some implications of this decision. We establish a nonlocal-to-local convergence result, showing that these operators coincide with the classical derivative as the nonlocality vanishes. We also provide a new integration by parts result, a characterization of the compactness of these nonlocal operators, an implication for the nonlocal Poincar\'e inequality, and a variety of examples. This work establishes several key results needed to analyze nonlocal variational problems, given in Part II~\cite{PartII}. We hope that this paper can serve as a relatively gentle introduction to the field of nonlocal modeling.
	\end{abstract}
	
	\section{Introduction}\label{sect: intro}
	Irregular behavior and multi-scale structures are ubiquitous in real-world models. Classical models that rely on differential operators implicitly require some regularity for their solutions. Over the past couple of decades, there has been a surge of interest in developing models involving nonlocal operators---in particular, operators with a convolution-like structure. Their general form for this paper is
	\begin{equation}\label{Eq:Prototype}
		\D_\delta u(x)=\int_{\R^d}[u(x+z)-u(x)]\mu_\delta(z)\dif z.
	\end{equation}
	Long-range interactions are captured by the kernel $\mu_\delta:\R^d\to\R$, which we assume to be integrable. Like a classical differential operator, with appropriate structural assumptions on $\mu_\delta$, the operator $\D_\delta$ will measure properties of function $u:\R^d\to\R$ related to its rate of change. Unlike a differential operator, however, $\D_\delta u(x)$ is well-defined even if $u$ fails to be differentiable at $x\in\R^d$. The \emph{horizon parameter} $\delta>0$ identifies the support of $\mu_\delta$, and determines the spatial scale of relative changes in $u$ captured by $\D_\delta u$. This nonlocal nature of $\D_\delta$ provides a natural mechanism for modeling phenomenon exhibiting irregular phenomena, dependent on long-range interactions, or involving multi-scale physics.
	
	Due to this flexibility, nonlocal operators have been successfully employed to model a rapidly expanding range of fields. These include swarming~\cite{mogilner_non-local_1999, gal2013global}, change in public sentiment~\cite{shomberg2021modeling}, composite-based aircraft (as Boeing's 787 Dreamliner~\cite{hu2012peridynamic}), fracture mechanics~\cite{ha2011characteristics}, image processing~\cite{images2,images1}, and phase separation~\cite{phases}. For a comprehensive introduction to nonlocal modeling, we refer to Quian Du's book~\cite{du_nonlocal_2019}.
	
	Some additional motivation for considering operators like $\D_\delta$ is provided in Section~\ref{sect: motiv}. We note that second-order nonlocal derivatives have been more extensively studied. Although we focus on operators that can be viewed as first-order nonlocal derivatives, many of the following results do not require this interpretation. In particular, except where explicitly stated, we do not assume that $\mu_\delta$ satisfies the hypotheses of Theorem~\ref{thm: convergence to classical derivative}. Hence, many of the tools we present are relevant to nonlocal operators that do not converge to any derivative, or those that converge to a higher- or mixed-order differential operator. 
	
	Because we are motivated by analyzing first-order nonlocal derivatives, we now comment on some of the existing work on the subject. Two-point nonlocal gradients have been studied in~\cite{gunzburger_nonlocal_2010, Du_Gunzberger}, and have been applied extensively in image processing~\cite{gilboa_nonlocal_2009}. However, we focus on weighted, one-point nonlocal operators. Operators of a similar structure have been analyzed in~\cite{mengesha_localization_2015, shankar_nonlocal_2016, mengesha_characterization_2016, du_analysis_2017, shieh_new_2018, delia_helmholtz-hodge_2020, delia_towards_2021, bellido_non-local_2023, haar_new_2022, delia_connections_2022,cueto_variational_2023}. Our work is distinct in that (1) we do not assume that the kernel $\mu_\delta$ is antisymmetric, (2) the kernel $\mu_\delta$ is integrable with support in the ball $\overline{B_\delta(0)}$, and (3) we work on arbitrary bounded domains. See Section~\ref{sect: symm} for further discussion on symmetry of $\mu_\delta$. Notably, the integrability requirement rules out the strongly singular kernels used to define fractional derivatives and ensures that $\D_\delta$ is a bounded operator on all of $L^p$. Nonlocal operators of this sort are often cited as an alternative to fractional derivatives, since they can be applied on arbitrary bounded domains as opposed to the whole space $\R^d$. In fact, nonlocal derivatives can be seen as a bridge between the classical and fractional calculus. In appropriate conditions, the fractional case can be recovered in the limit $\delta\to\infty$ and the classical derivative is seen as the limit $\delta\to0$. These ideas are developed in~\cite{delia_fractional_2013, tian_asymptotically_2016, delia_towards_2021, delia_connections_2022}, and a concise comparison of nonlocal, fractional, and local gradients is given in~\cite{cueto_variational_2023}.

	One goal of nonlocal modeling is to extend some familiar ideas from the classical, local theory of calculus and differential equations. In the nonlocal setting, there are few restrictions on the domain, as well as low regularity assumptions on the functions we work with. To simplify the presentation, we focus on scalar-valued functions and scalar-valued kernels. However, our results can be easily extended to higher dimensions. This work pulls together some tools that will be put to use in Part II~\cite{PartII}, where we discuss variational problems involving nonlocal differential operators.
	
	\subsection{Contributions}
	The main contributions of this paper are as follows
	\begin{enumerate}
		\item \textit{A New Convergence Result:} As discussed in Section~\ref{sect: motiv}, $\D_\delta u$ can be seen as a nonlocal generalization of a directional derivative. This is made precise by Theorem~\ref{thm: convergence to classical derivative}, which provides a more general form of operator convergence than exists in the literature. We do not impose any symmetry restrictions on the kernel, assume more flexible scaling and concentration hypotheses, and provide explicit estimates on the rate of convergence. Depending on the regularity of $u$, $\D_\delta u$ can be seen as a nonlocal approximation of the distributional, weak, or strong derivative of $u$. Section~\ref{sect: symm} gives some additional context for the significance of Theorem~\ref{thm: convergence to classical derivative}, and some examples are produced in Section~\ref{sect: ex}.
		
		\item \textit{A New Nonlocal Integration By Parts:} Section~\ref{sect: adjoint} establishes the adjoint of the operator $\D_\delta$ when acting on $L^p(\Omega_\delta)$ for $p\in[1,\infty)$. We also give a general expression for integrating by parts when using nonlocal derivatives in Theorem \ref{thm: by parts}. To our knowledge, this is the first such result on bounded domains for these sorts of operators without symmetry assumptions on the kernel. Nonlocal analogs of Green's Identity, the Divergence Theorem, and the familiar adjoint relationship between the gradient and divergence are immediate corollaries of Theorems \ref{thm: adjoint} and \ref{thm: by parts}.
		
		\item \textit{Exploration of Operator Theoretic Properties:} Theorem~\ref{thm: cpt} gives a characterization of the compactness of the operator $\D_\delta$. As a consequence, Corollary~\ref{coro: no poincare} shows that if $\mu_\delta$ is integrable and antisymmetric, then there cannot be a Poincar\'{e}-type inequality for $\D_\delta$. Implications to variational problems will be explored in a future work~\cite{PartII}. Sections~\ref{sec: compact} and~\ref{sect: symm} also call attention to a theme running through the nonlocal literature: nonlocal operators with mean-free kernel have a much larger null space than the classical derivative. This introduces complications in modeling, theoretical development, and numerical implementation. This further demonstrates the need to develop a general theory for nonlocal operators without (anti)symmetry assumptions.
	\end{enumerate}
	
	\subsection{Setup and Notation}\label{sec: setup}
	Throughout, $\mu_\delta$ is assumed to be a scalar-valued function in $L^1(\R^d)$. We typically assume that $\mathop{\mathrm{supp}}(\mu_\delta)\subseteq \overline{B_\delta(0)}$, where $\delta>0$ identifies the horizon, or interaction radius, of the operator $\D_\delta$. In the convergence theorem, \ref{thm: convergence to classical derivative}, the support of $\mu_\delta$ is allowed to be unbounded. In this case, $\delta$ roughly corresponds to the radius of the ball that contains ``most" of the mass of $\mu_\delta$; this notion is made precise where needed. A physical interpretation of the horizon $\delta$ is provided in~\cite{bobaru_meaning_2012, seleson_role_2011}. In the context of image processing, $\delta$ controls the size of the ``patches" or ``neighborhoods"~\cite{buades_non-local_2011}.
	
	We take $d\in\N$ to be the dimension of the ambient space, and $\Omega$ is an open, bounded subset of $\R^d$. Because of the positive interaction radius, boundary conditions for nonlocal problems are imposed on a set of positive measure, so called ``volume constraints" (see for example~\cite{Du_Gunzberger, delia_towards_2021}). Hence, we need a nonlocal analog of the boundary, which we call a \textit{collar}.
	\begin{definition}\label{def: collar}
		Let $\Omega\subseteq\R^d$ be a given open, bounded domain. If $\delta>0$ is fixed without specifying a kernel, then the \textit{collar of width $\delta$} is denoted by $\Gamma_\delta$ and defined by $\Gamma_\delta\defn \left(\bigcup_{x\in\Omega} \overline{B_\delta(x)} \right)\setminus\Omega$. In general, given a specific kernel $\mu_\delta$, one can instead work with a potentially smaller domain:  $\Gamma_\delta\defn\bigcup_{x\in\Omega}\supp(\mu_\delta(\cdot-x))\setminus\Omega$.
	\end{definition} 
	\begin{definition}\label{def: Omega_delta}
		The setting for a nonlocal problem is typically $\Omega_\delta\defn \Omega\cup\Gamma_\delta$. The value of $\D_\delta(x)$ incorporates information within a ball of radius $\delta$. Thus, to ensure it is well-defined throughout $\Omega$, we ''attach'' a collar of width $\delta$ as a nonlocal boundary. Occasionally we will also need to reference an \textit{inner collar} $\Gamma_{-\delta}$, which is a strip of width $\delta$ inside $\Omega$, running along the boundary. Note that there is no necessary connection between the regularity or topological properties of $\Omega$ and $\Omega_\delta$.  A two-dimensional example is shown in Figure~\ref{fig: nonlocal domain}.
	\end{definition}
	
	\begin{figure}[!htbp]
		\centering
		\includegraphics[width=.5\textwidth]{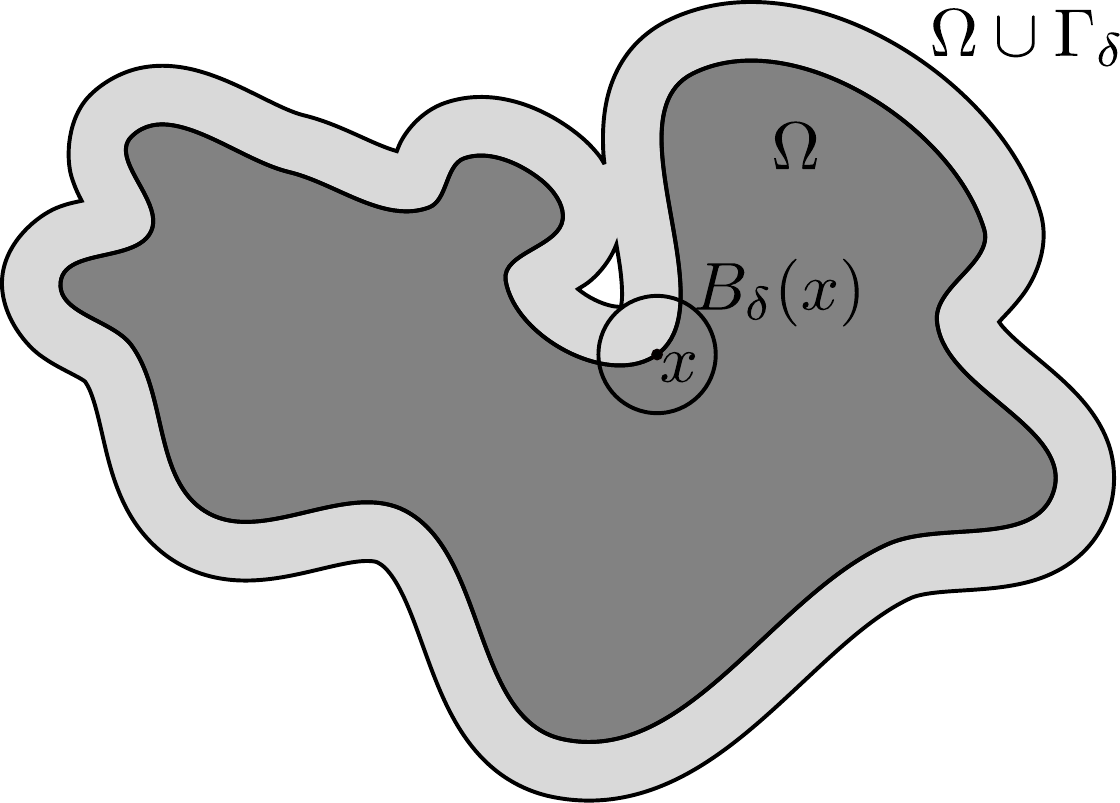}
		\caption{Domain $\Omega$ (dark grey) with collar $\Gamma_\delta$ (light grey).}

		\label{fig: nonlocal domain}
	\end{figure}
	
	\begin{definition}\label{def: inner collar}
		Given $\delta>0$, we set $\Gamma_{-\delta}\defn \set{x\in\Omega: d(x,\partial\Omega),\le \delta}.$ We also write $\Omega_{-\delta}$ to denote the set $\Omega\setminus \Gamma_{-\delta}$, which is the set of all interior points that are at least $\delta$ units away from the boundary.
	\end{definition}
	
	Typically, $p\in[1,\infty]$ denotes an exponent and $p'$ denotes its H\"older conjugate, so that $1/p + 1/p' = 1$. We use $L^p\left(\Omega\right)^m$ to denote vector-valued $L^p$ functions, while $L^p\left(\Omega\right)$ implies that the functions output real numbers. In this work, we focus on scalar-valued functions. Hence, we always assume that $u$ is scalar-valued, unless explicitly stated otherwise. We write $\vec u$ in the few cases where vector fields arise, to emphasize that the output is a vector.
	
	The central subject of this work is the nonlocal directional derivative. Because we only consider integrable kernels, it follows immediately from Young's inequality for convolutions that $\D_\delta u$ is well defined for any $u\in L^p(\Omega_\delta)$.
	
	\begin{definition}
		Pick any $\delta>0$, $p\in[1,\infty]$, and $u\in L^p\left(\Omega_\delta\right)$. For some $\mu_{\delta}\in L^1(\R^d)$, the \textit{nonlocal (directional) derivative} is the operator $\D_\delta: L^p(\Omega_\delta)\to L^p(\Omega)$, given by
		\begin{equation}\label{eqt: define D}
			\D_{\delta} u(x) \defn \int_{\R^d} [u(x+z) - u(x)]\mu_{\delta}(z)\dif z.
		\end{equation}
	\end{definition}
	
	\begin{definition}\label{def: zero on collar}
		In many cases it is useful to work on a closed subspace, $L^p_0\left(\Omega_\delta\right)$, consisting of all functions that are zero on the collar:
		\begin{equation}
			L^p_0\left(\Omega_\delta\right)\defn \set{u\in L^p\left(\Omega_\delta\right): u=0 \text{ a.e. in } \Gamma_\delta}.
		\end{equation}
	\end{definition}
	
	\begin{remark}
		When the domain is clear, we write $L^p_0$ instead of $L^p_0\left(\Omega_\delta\right)$. Note the similarity to the function space $W^{1,p}_0\left(\Omega\right)$, which can roughly be thought of as functions that are zero on the boundary. In our setting, the notion of boundary is replaced with nonlocal analog: the collar, or ``fat boundary."
	\end{remark}
	
	Note that the domain of $\D_\delta u$ is smaller than the domain of $u$, because the nonlocal derivative needs to ``see'' a ball of radius $\delta$ around every point in $\Omega$. Hence, $u$ is defined on $\Omega_\delta$ while $\D_\delta u$ is only defined on $\Omega$. This means $\D_\delta$ acts between two distinct Banach spaces, which can be inconvenient and motivates the following definitions.

	\begin{definition}\label{rmk: isometric isomorphism}
		The space $L^p_0$ introduced in Definition~\ref{def: zero on collar} is a closed subspace of $L^p\left(\Omega_\delta\right)$. It is also isometrically isomorphic to $L^p\left(\Omega\right)$, via the extension-by-zero operator $E: L^p\left(\Omega\right)\to L^p_0$ defined as
		\[
		Eu(x)\defn \begin{cases}
			u(x), & x\in\Omega\\
			0, & x\in\Gamma_\delta.
		\end{cases}
		\]
		We also let $P: L^p\left(\Omega_\delta\right)\to L^p\left(\Omega\right)$ denote the operator defined by $Pu(x)= u(x)$ for all $x\in \Omega$.
	\end{definition}
	
	For any $x\in\Omega$, we see that $EPu(x) = u(x)$, while for all $x\in\Gamma_\delta$, we have $EPu(x)=0$. This means $EP$ acts as the identity operator when its domain is restricted to $L^p_0$. Thus, when restricted to $L^p_0$, the operator $P$ acts as the inverse of $E$ and is an isometric isomorphism between $L^p_0$ and $L^p\left(\Omega\right)$. In this sense, $P$ is a projection operator from $L^p\left(\Omega_\delta\right)$ onto the closed subspace $L^p\left(\Omega\right)\cong L^p_0$.

	Using the extension $E$ and projection $P$ allows us to change the function spaces $\D_\delta$ acts between. The point of this is to apply some tools from the theory of bounded linear operators acting on a single Banach space.

	\begin{definition}
		We define the \textit{extended nonlocal derivative} $\overline{\D_\delta}:L^p\left(\Omega_\delta\right)\to L^p\left(\Omega_\delta\right)$ as $\overline{\D_\delta}\defn E\D_\delta$. This can also be written as a single integral operator. If $u\in L^p\left(\Omega_\delta\right)$ and $x\in\Omega_\delta$, then
		\begin{equation}\label{eqt: extended domain}
			\overline{D_\delta} u(x)\defn \int_{\Omega_\delta} [u(y)-u(x)]\mu_\delta(y-x)\chi_{B_\delta(x)}(y)\chi_{\Omega}(x) \dif y.
		\end{equation}
		For points in $\Omega$, equation~\eqref{eqt: extended domain} is equivalent to the first definition,~\eqref{eqt: define D}. The only difference is that $\overline{\D_\delta} u(x)=0$ for all $x\in\Gamma_\delta$, while $\D_\delta u$ is undefined for such points.
	\end{definition}
	
	Viewing $\D_\delta$ as a nonlocal directional derivative naturally leads us to construct nonlocal analogs to other familiar differential operators. We use $\mathscr A_\delta$ to denote a first-order, linear nonlocal differential operator, $\G_\delta$ to denote a nonlocal gradient, and $\mathscr D_\delta$ to denote a nonlocal divergence. Precise formulations are given in Definition~\ref{def: nonlocal differential operators}. As shown in Propositions~\ref{prop: convergence of vect ops} and~\ref{prop: div grad adj}, these operators can be seen as nonlocal analogs of the familiar gradient and divergence.

	\section{Basic Properties of the Nonlocal Directional Derivative}\label{sect: Operator}

	In this section, we explore some basic properties of the operator $\D_\delta$. In particular, we show that it acts as a bounded linear operator on $L^p\left(\Omega_\delta\right)$. Then, under some general hypotheses on $\mu_\delta$, we show that $\D_\delta$ converges to a directional derivative as $\delta$ shrinks. After presenting some examples, we then define nonlocal analogs of first-order differential operators from vector calculus. These are not new operators, and in fact versions of them have been well-studied. See Sections~\ref{sect: intro} and~\ref{sect: symm} for additional comparison with existing results.

	\subsection{Motivation}\label{sect: motiv}
	One way to interpret the definition of $\D_\delta$ in equation~\eqref{eqt: define D} is as an average of difference quotients. The classical derivative is a local operator since it involves taking a limit at a point. That is, $\grad \mid_{x_0}$ only depends on the values of the input function $u$ in an infinitesimal neighbor of the point $x_0$. We want to develop a nonlocal analog that still captures information about how the function $u$ changes in a small region about $x$. One approach is to take a weighted average of the difference quotient, rather than taking a limit:
	\[
	\D_\delta u (x) = \int_{B_\delta(0)}\frac{u(x+z)-u(x)}{z} \rho_\delta(z)\dif z,
	\]
	where $\rho_\delta$ is some weighting function or probability distribution. If we want $\D_\delta$ to approach the local derivative, perhaps $\rho_\delta$ should weight points closer to $x$ more heavily. Though, strictly speaking, this is not required by the theory. Setting $\mu_\delta(z) = \frac1z \rho_\delta(z)$, we obtain the definition of $\D_\delta$ given in equation~\eqref{eqt: define D}. If $\rho_\delta$ is a probability distribution, then $\int \rho_\delta \dif z=1$. In terms of our notation, this amounts to the requirement that
	\begin{equation}\label{eqt: probability scaling}
		\int_{B_\delta(0)} z\mu_\delta(z)\dif z = 1.
	\end{equation}
	In Theorem~\ref{thm: convergence to classical derivative}, the limiting behavior of this first moment determines the differential operator $\D_\delta$ approaches, as $\delta\to0$.
	
	Now we present a brief formal argument to demonstrate the relation between $\D_\delta$ and a derivative. With $d=1$ and $\Omega=(a,b)$, assume $u: \Omega_\delta\to\R$ is analytic. At $x\in(a,b)$, we can introduce the Taylor expansion of $u$ in the definition of $\D_\delta$ to obtain
	\begin{multline}
		\D_{\delta} u(x) = u'(x)\int_{-\delta}^{\delta} z \mu_{\delta}(z)\dif z +\frac{u''(x)}{2}\int_{-\delta}^{\delta}  z^2 \mu_{\delta}(z)\dif z+\frac{u'''(x)}{6}\int_{-\delta}^{\delta}  z^3 \mu_{\delta}(z)\dif z \\ +\frac{u^{(4)}(x)}{4!}\int_{-\delta}^{\delta}  z^4 \mu_{\delta}(z)\dif z+\frac{u^{(5)}(x)}{5!}\int_{-\delta}^{\delta}  z^5 \mu_{\delta}(z)\dif z+\dots
	\end{multline}
	If $\{\mu_\delta\}_{\delta>0}$ is bounded in $L^1(\R)$ and $\int_{-\delta}^{\delta} z \mu_{\delta}(z)\dif z \to 1$, then the first term in the above expansion dominates and
	$\D_{\delta} u(x) \to u'(x)$.
	
	We can also motivate the convergence $\D_\delta\to \od{}{x}$ in terms of mollification (see, for example, Section 5.1 of~\cite{Du_Gunzberger}). In this setting, we view $\set{\mu_\delta}_{\delta>0}$ as a derivative of an approximation of the identity. 
	For example, suppose $u$ is smooth enough and set $\mu_\delta \defn -\eta_\delta'$, where $\eta_\delta$ is appropriately scaled standard mollifier: 	\[
	\eta_\delta(z) = \begin{dcases}
		C_\delta e^{\frac{-1}{1-\abs{z/\delta}^2}}, & \abs{z/\delta}<1\\
		0 & \abs{z/\delta}\ge 1.
	\end{dcases}
	\]
	The function $\eta$ is symmetric, but its derivative $\eta_\delta'$ is antisymmetric. One can quickly check that the corresponding nonlocal derivative takes the form of a convolution. Using the fact that derivatives of convolutions can be applied to either term, we have
	$
	\D_\delta u = u\ast (-\mu_\delta) =  u\ast \eta_\delta'=u'\ast \eta_\delta.
	$
	Since the functions $\eta_\delta$ form an approximation of the identity, we expect $u'\ast \eta_\delta$ to converge to $u'$ as $\delta$ shrinks. This idea can be made precise with the language of distributions. An alternative approach is given in Theorem~\ref{thm: convergence to classical derivative}.
	
	\subsection{Alternate Expressions for $\D_\delta$} \label{sect: alt exp}	
	Recall the definition of the nonlocal directional derivative: $
	\D_\delta u (x)\defn \int_{B_\delta(0)} [u(x+z)-u(x)]\mu_\delta(z)\dif z.
	$
	In this form, the domain of integration is independent of $x$, which can be helpful in certain contexts. But there are several equivalent ways of expressing this operator. Via a change of variables, we have
	\begin{equation}\label{eqt: def of Ddelta change vars}
		\D_\delta u (x)= \int_{B_\delta(x)} [u(y)-u(x)]\mu_\delta(y-x)\dif y.
	\end{equation} 
	We assume that $\mu_\delta:\R^d\to\R$ has support contained in $\overline{B_\delta(0)}$. To make this explicit, we sometimes include a characteristic function inside the integral.
	
	It is often helpful to decompose $\D_\delta$ as an integral operator plus a scalar multiple of $P$, the projection defined in Remark~\ref{rmk: isometric isomorphism}:
	\begin{gather}\label{eqt: decompose Ddelta}
		\D_\delta = I_\delta - \overline{\mu_\delta} P ,\ \  \text{ where } \ 		\overline{\mu_\delta}\defn \int_{B_\delta(0)} \mu_\delta(z)\dif z, \text{ and}\\ 
		I_\delta: L^p\left(\Omega_\delta\right)\to L^p\left(\Omega\right), \ \ I_\delta u(x) \defn \int_{\Omega_\delta} u(y)\mu_\delta(y-x)\dif y.
	\end{gather}
	
	\begin{remark}\label{rmk: I is crosscor}
		The integral operator $I_\delta$ can be identified as a cross-correlation: $I_\delta u = u\star \mu_\delta$. With $s(x)=-x$, we may instead write $I_\delta u= u\ast (\mu_\delta\circ s)$. Thus, the operator $I_\delta$ inherits many of the familiar properties of convolutions.
	\end{remark}
	
	\begin{definition}
		We also define the extended operator $\overline{I}_\delta: L^p(\dom_\delta)\to L^p(\dom_\delta)$ by
		\[
		\overline{I_\delta}u(x)\defn EI_\delta u(x)= \int_{\Omega_\delta} u(y)\mu_\delta(y-x)\chi_{B_\delta(x)}(y)\chi_{\Omega}(x)\dif y.
		\]
	\end{definition}
	We observe that, using the operators $E$ and $P$, defined in Remark~\ref{rmk: isometric isomorphism} we can write
	\begin{gather}\label{eqt: decompose extended domain}
		\overline{\D_\delta} = \overline{I_\delta} - \overline{\mu_\delta} EP
	\end{gather}
	where the extended nonlocal derivative $\overline{\D_\delta}$ was defined in equation~\eqref{eqt: extended domain}.
	
	The integral operator $\overline{I_\delta}$ in equation~\eqref{eqt: decompose extended domain} can be written as
	\begin{equation}\label{eqt: kernel of int op}
		\overline{I_\delta} u(x) = \int_{\Omega_\delta} u(y)K(x,y)\dif y, \text{ where } K(x,y)\defn \mu_\delta(y-x)\chi_{B_\delta(x)}(y)\chi_{\Omega}(x).
	\end{equation}
	This perspective allows us to apply the well-developed theory of integral operators to $\overline{\D_\delta}$. We can also interpret the other term in in the definition of $\overline{\D_\delta}$ as a multiplication operator: $\overline{\mu_\delta}EP= M_{g}$ where $g = \overline{\mu_\delta}\chi_{\Omega}$. Hence, the extended nonlocal derivative is a combination of an integral operator and a multiplication operator: $\overline{\D_\delta} = \overline{I_\delta} - M_{g}$.
	
	\subsection{Continuity and Convergence to the Classical Derivative}\label{sect: convergence}
	The operator $\D_\delta$ is defined on all of $L^p(\Omega_\delta)$ and in fact is a continuous map. 
	\begin{Prop}\label{prop: D is bdd}
		For any $p\in[1,\infty]$, $\D_\delta: L^p(\Omega_\delta)\to L^p\left(\Omega\right)$ and $\overline{\D_\delta}: L^p(\Omega_\delta)\to L^p(\Omega_\delta)$ are bounded linear operators. 
	\end{Prop}
	
	\begin{proof}
		Linearity follows from the linearity of integration. So we only check that the nonlocal derivative is bounded. to this end, fix any $u\in L^p(\Omega_\delta)$. Applying Minkowski's integral inequality to the definition of $\D_\delta$, we obtain \[\norm{\D_\delta u}_{L^p(\Omega)}\le \int_{\R^d} \abs{\mu_\delta(z)} \norm{u(\cdot + z) - u(\cdot)}_{L^p(\Omega)}\dif z \le 2\norm{u}_{L^p(\Omega_\delta)}\norm{\mu_\delta}_{L^1(\R^d)}. \]
		Hence, the operator norm of $\D_\delta$ is bounded above by $2\norm{\mu_\delta}_{L^1(\R^d)}$.
		
		Since $\overline{\D_\delta} = E\D_\delta$ and $E$ is an isometry, we obtain the same bound on the norm of the extended nonlocal derivative.
	\end{proof}
	
	As indicated above, $\delta>0$ identifies the scale at which information from $u$ affects $\D_\delta u$. Thus, it determines how ``nonlocal'' of an operator $\D_\delta$ is. Many nonlocal models are expected to replicate a classical local model when $u$ is sufficiently regular. This requires, in particular, $\D_\delta$ to converge to a given differential operator $\mathsf{D}$ as $\delta\to0^+$. The next Theorem provides provide structural assumptions on $\mu_\delta$ that ensure this convergence. Moreover, if $\mu_\delta$ has compact support, we provide lower bounds for the rate of convergence. This makes precise the intuition presented in Section~\ref{sect: motiv}.
	
	There are a variety of related results existing in the literature. For example, Lemma 3.6 in~\cite{du_nonlocal_2019}, Theorems 3.2 and 3.12 in~\cite{mengesha_characterization_2016}, Theorems 1.1 and 3.2 in~\cite{mengesha_localization_2015}, or Theorem 2.1 in~\cite{shankar_nonlocal_2016}. For a more precise estimate on the rate of convergence via the Hardy-Littlewood maximal operator, see \cite{larios_note_2024}.
	
	\begin{thm}[Operator Convergence]\label{thm: convergence to classical derivative}
		Let $\{\mu_\delta\}_{\delta>0}\subseteq L^1(\R^d)$ and $a$ be a fixed unit vector in $\R^d$. Assume there exists a nondecreasing $\omega\in C([0,\infty);[0,\infty))$ such that $\omega(0)=0$. Later, we refer to the following hypotheses:
		\begin{enumerate}
			\item[(i)] \textbf{$L^1$-bounded moments}: There exists $M<\infty$ such that
			\begin{equation*}\label{eqt: first moment is bounded}
				\int_{\R^d} \abs{\mu_{\delta}(z)}\abs{z}\dif z < M
				\quad\text{ for all }\delta>0;
			\end{equation*}
			\item[(ii)] \textbf{Convergence of moments}:
			\begin{equation*}\label{eqt: scaling 1}
				\left|
				\int_{\R^d} \mu_{\delta}(z)z \dif z-a\right|\le\omega(\delta);
			\end{equation*}
			\item[(iii')] \textbf{Concentration}:
			\begin{equation*}\label{eqt: concentration}
				\int_{\R^d\setminus B_\delta(0)}|\mu_\delta(z)||z|\dif z\le\omega(\delta),
				\quad\text{ for all }\delta>0;
			\end{equation*}
			\item[(iii'')] \textbf{Compact support}: $\supp(\mu_\delta)\subseteq\overline{B}_\delta(0)$.
		\end{enumerate}
		Let $1\le p<\infty$ be given. Then the following hold.
		\begin{enumerate}
			\item[(a)] If $u\in L^p(\R^d)$ and (i), (ii), and (iii') all hold, then $\D_\delta u\in L^p(\R^d)$, and for each $\varphi\in W^{1,p'}(\R^d)$ and $\eps>0$, there exists $\delta_\varepsilon>0$ such that
			\[
			\left|\int_{\R^d}\lp\D_\delta u(x)\rp\varphi(x)\dif x
			+\int_{\R^d} u(x)\lp a\bdot\nabla\varphi (x)\rp\dif x\right|<\varepsilon\|u\|_{L^p(\R^d)}.
			\]
			Here, $\bdot$ denotes the Euclidean dot product.
			\item[(b)] If $u\in W^{1,p}(\R^d)$ and (i), (ii), and (iii') all hold, then for each $\eps>0$, there exists $\delta_\eps>0$ such that,
			\[
			\|\D_\delta u-a\bdot\grad u\|_{L^p(\R^d)}<\eps
			\quad \text{ for all }0<\delta<\delta_\eps.
			\]
			\item[(c)] If there exists $0<\alpha\le 1$ such that both $\omega,\nabla u\in C^{0,\alpha}(\R^d)$ and (i), (ii), and (iii'') hold, then for each $x\in\R^d$, there exists $C<\infty$ such that
			\[
			|\D_\delta u(x)-a\bdot\grad u(x)|<C\delta^\alpha
			\quad\text{ for all }\delta>0.
			\]
		\end{enumerate}
	\end{thm}

	\begin{proof}
		\textbf{Part (a)}: Let $u\in L^p(\R^d)$ and $\varphi\in C^\infty_{\text{c}}(\R^d)$ be given. The integrability of $\D_\delta u$ follows from Young's convolution inequality. For brevity, put
		\[
		\mathsf{E}=\int_{\R^d}\D_\delta u(x)\varphi(x)\dif x+\int_{\R^d}u(x) a\bdot\nabla\varphi(x)\dif x.
		\]
		By the fundamental theorem of calculus,
		\begin{equation*}\label{eqt: FTC}
			\varphi(x-z)-\varphi(x)=-\nabla \varphi(x)\bdot z
			-\int_0^1\left[\nabla \varphi(x-sz)-\nabla \varphi(x)\right]\bdot z\dif s,
			\quad\text{ for all }z\in\R^d.
		\end{equation*}
		With $x\in\dom$, from the definition of $\D_\delta u\in L^p(\R^d)$, we may write
		\begin{align*}
			\int_{\R^d}\D_\delta u(x)\varphi(x)\dif x
			=&
			\int_{\R^d}\int_{\R^d}\ls u(x+z)-u(x)\rs\varphi(x)\mu_\delta(z)\dif z\dif x\\
			=&
			\int_{\R^d}\int_{\R^d}u(x)\ls\varphi(x-z)-\varphi(x)\rs\mu_\delta(z)\dif z\dif x\\
			=&
			-\int_{\R^d} u(x)\lp\int_{\R^d}\mu_\delta(z) z\dif z\rp\bdot\nabla\varphi(x)\dif x\\
			&\qqquad
			-\int_{\R^d}\int_{\R^d}\int_0^1 u(x)\ls\nabla\varphi(x-sz)-\nabla\varphi(x)\rs\mu_\delta(z)\bdot z\dif s\dif z\dif x.
		\end{align*}
		Here, we applied Fubini's theorem (twice) and a change of variables. Thus,
		\begin{multline*}
			\abs{\mathsf{E}}
			\le
			\int_{\R^d}\left|\int_{\R^d}\mu_\delta(z)z\dif z-a\right||u(x)||\nabla \varphi(x)|\dif x\\
			+\int_{\R^d}\int_{\R^d}\int_0^1
			|u(x)||\nabla\varphi(x+sz)-\nabla \varphi(x)||\mu_\delta(z)||z|\dif s\dif z\dif x
		\end{multline*}
		Fix $\eps>0$. The convergence of the first moments of $\mu_\delta$ implies that there exists a $\delta'>0$ such that 
		\[
		\left|\int_{\R^d}\mu_\delta(z)z\dif z-a\right|
		\le\frac{\varepsilon}{1+3\|\nabla\varphi\|_{L^{p'}(\R^d)}},
		\quad\text{ for all }0<\delta<\delta'.
		\]
		Using this, Fubini's Theorem and H\"older's inequality, we obtain the bound
		\begin{equation*}
			\abs{\mathsf{E}}
			\le \frac{\varepsilon}{3}\|u\|_{L^p(\R^d)}
			+\int_0^1\int_{\R^d}|\mu_\delta(z)||z|
			\|u\|_{L^p(\R^d)}
			\|\nabla\varphi(\cdot+sz)-\nabla\varphi(\cdot)\|_{L^{p'}(\R^d)}\dif z\dif s.
		\end{equation*}
		Since $\nabla\varphi$ is continuous with respect to the $L^{p'}$-norm, there exists $r>0$ such that
		\[
		\|\nabla\varphi(\cdot+h)-\nabla\varphi(\cdot)\|_{L^{p'}(\R^d)}
		\le\frac{\varepsilon}{1+3M},
		\quad\text{ for all }h\in B_{r}(0).
		\]
		Here $M<\infty$ is the uniform $L^1$-bound from \eqref{eqt: first moment is bounded}. Furthermore, the concentration assumption provides a $0<\delta_{\varepsilon}\le\delta'$ such that
		\[
		\int_{\R^d\setminus B_\rho(0)}|\mu_\delta(z)||z|\dif z
		\le
		\omega(\delta)\le\omega(\delta_\eps)<
		\frac{\varepsilon}{1+6\|\nabla\varphi\|_{L^{p'}(\R^d)}},
		\quad\text{ for all }0<\delta<\delta_{\varepsilon}.
		\]
		It follows that, for all $0<\delta<\delta_\varepsilon$,
		\begin{align*}
			\abs{\mathsf{E}}\le&\frac{\varepsilon}{3}\|u\|_{L^p(\R^d)}+\int_0^1\int_{\R^d}|\mu_\delta(z)||z|
			\|u\|_{L^p(\R^d)}
			\|\nabla\varphi(\cdot+sz)-\nabla\varphi(\cdot)\|_{L^{p'}(\R^d)}\dif z\dif s\\
			\le&
			\frac{\varepsilon}{3}\|u\|_{L^p(\R^d)}
			+\lp\frac{\eps}{1+3M}\rp\|u\|_{L^p(\R^d)}\int_{B_{\rho}(0)}
			|\mu_\delta(z)||z|\dif z\dif s\\
			&\qqqquad\qqqquad+
			2\|u\|_{L^p(\R^d)}\|\nabla\varphi\|_{L^{p'}(\R^d)}\int_{\R^d\setminus B_\rho(0)}
			|\mu_\delta(z)||z|\dif z\\
			\le&
			\varepsilon\|u\|_{L^p(\R^d)}.
		\end{align*}
		Using density of $C^\infty_{\text{c}}(\R^d)$ we can extend the convergence to any $\phi\in W^{1,p'}(\R^d)$. This establishes (a).
		
		\textbf{Part (b)}: We may assume that $u\in C^1(\R^d)$ and $\grad u\in L^p(\R^d)$. Fix $\eps>0$. Given $x\in\R^d$, similar to part (a), we use the fundamental theorem of calculus to obtain
		\begin{equation}\label{eqt: Bound}
			\left|\D_\delta u(x)-a\bdot\nabla u(x)\right|
			\le
			|\nabla u(x)|\left|\int_{\R^d}\mu_\delta(z)z\dif z-a\right|
			+\left|\int_{\R^d}\int_0^1
			\left[\nabla u(x+sz)-\nabla u(x)\right]\bdot\mu_\delta(z)z\dif s\dif z\right|.
		\end{equation}
		Using this and Minkowski's integral inequality, we obtain the bound
		\begin{multline}\label{eqt: conv bnd 1}
			\|\D_\delta u-a \bdot\nabla u\|_{L^p(\R^d)}
			\le \int_{\R_d}\left|\int_{\R^d}\mu_\delta(z)z\dif z-a\right||\nabla u(x)|\dif x\\
			+\int_0^1\int_{\R^d}|\mu_\delta(z)||z|
			\|\nabla u(\cdot+sz)-\nabla u(\cdot)\|_{L^p(\R^d)}\dif z\dif s.
		\end{multline}
		From assumptions (i), (ii), and (iii'), we also deduce there exists $\delta_{\eps}>0$ such that 
		\[
		\left|\int_{\R^d}\mu_\delta(z)z\dif z-a\right|
		\le\frac{\varepsilon}{1+3\|\nabla u\|_{L^p(\R^d)}}
		\]
		and a $r>0$ such that
		\[
		\|\nabla u(\cdot+h)-\nabla u(\cdot)\|_{L^p(\R^d)}
		\le\frac{\varepsilon}{1+3M}
		\,\text{ and },
		\int_{\R^d\setminus B_r(0)}|\mu_\delta(x)||z|\dif z
		\le
		\omega(\delta)\le\omega(\delta_\eps)<
		\frac{\varepsilon}{1+6\|\nabla u\|_{L^p(\R^d)}},
		\]
		for all $h\in B_r(0)$ and $0<\delta<\delta_\eps$.
		It follows that, for all $0<\delta<\delta_\varepsilon$,
		\begin{align*}
			\int_0^1\int_{\R^d}|\mu_\delta(z)||z|
			\|\nabla u(\cdot+sz)-\nabla u(\cdot)\|_{L^p(\R^d)}\dif z\dif s
			\le&
			M\int_0^1\int_{B_{\rho}(0)}
			\|\nabla u(\cdot+sz)-\nabla u(\cdot)\|_{L^p(\R^d)}\dif z\dif s\\
			&+
			2\|\nabla u\|_{L^p(\R^d)}\int_{\R^d\setminus B_\rho(0)}
			|\mu_\delta(z)||z|\dif z\\
			\le&\frac{2\varepsilon}{3}.
		\end{align*}
		Incorporating this into~\eqref{eqt: conv bnd 1} yields the statement in (b).

		\textbf{Part (c)}: First observe that $\supp(\mu_\delta)\subseteq\overline{B}_\delta(0)$ implies
		\[
		\int_{\R^d\setminus B_\delta(0)}\int_0^1
		\left|\nabla u(x+sz)-\nabla u(x)\right||\mu_\delta(z)||z|\dif z=0
		\]
		Let $x\in\R^d$ and $\delta>0$ be given. From~\eqref{eqt: Bound} and the convergence of moments
		\begin{equation*}
			\left|\D_\delta u(x)-a\bdot\nabla u(x)\right|
			\le
			|\nabla u(x)|\omega(\delta)
			+
			\int_{B_\delta(0)}\int_0^1
			\left|\nabla u(x+sz)-\nabla u(x)\right||\mu_\delta(z)||z|\dif z.
		\end{equation*}
		As both $\omega$ and $\grad u$ are in $C^{0,\alpha}$, this implies that, for some $K<\infty$,
		\[
		\left|\D_\delta u(x)-a\bdot\nabla u(x)\right|
		\le
		|\nabla u(x)|\omega(\delta)
		+ K\delta^\alpha
		\int_{B_\delta(0)}|\mu_\delta(z)||z|\dif z \le     
		K\delta^\alpha\left(|\nabla u(x)| + M\right).
		\]
		For the last inequality, we used the $L^1$-bound in assumption (i).
	\end{proof}

	\begin{remark}\label{rmk: assume R wlog}
		In the previous theorem, the choice of ${\R^d}$ for the domain was primarily for convenience. Since we are only interested in behavior for small $\delta$, there is no need to consider all positive $\delta$. For example, the same argument holds if $u\in W^{1,p}\left(\Omega_C\right)$ (for any $C>0$) and we only consider kernels $\set{\mu_{\delta}}_{\delta\in(0,C)}$.
		
		Additionally, if we assume that $\Omega_C$ is an extension domain, then there exists $\bar u: \R^d\to \R$ such that $\bar u(x) = u(x)$ for all $x\in \Omega_\delta $ and $\norm{\bar u}_{W^{1,p}(\R^d)} \le C_u\norm{u}_{W^{1,p}\left(\Omega_\delta\right)}$ for some constant $C_u$~\cite[Theorem 1, Section 5.4]{Evans}. Then we apply the theorem as stated to this extended function. For a characterization of extension domains, see the works~\cite{jones_extension_1980, jones_quasiconformal_1981}. This class of domains is relatively large, and even includes fractal domains like the Koch snowflake.
		
		In this work, we focus on kernels with support contained in the closed ball $\overline{B_\delta(0)}$. So long as this holds and $\lim_{\delta\to0^+} \int_{B_\delta(0)} z\mu_\delta(z)\dif z = a$, then the concentration and convergence of moments assumptions are satisfied.
	\end{remark}
	
	\begin{remark}\label{rmk: higher order}
		Returning to the discussion at the beginning of Section~\ref{sect: motiv}, we see that choosing $\mu_\delta$ so that the higher order moments are zero ensures faster convergence $\D_\delta u \to a\bdot\nabla u$ for sufficiently smooth functions $\varphi$ or $u$. As illustrated in Example \ref{ex: poly}, given $k\in\mathbb{N}$ one can design a piecewise constant or polynomial kernel such that
		\begin{gather*}
			\int_{B_\delta(0)} \mu_{\delta}(z)z\dif z=a, \ \ 
			\int_{B_\delta(0)} \mu_{\delta}(z) z\otimes z \dif z=0,\ \ \dots, \ \,
			\text{ and }
			\int_{B_\delta(0)} \mu_{\delta}(z)
			\underbrace{z\otimes\cdots\otimes z}_{\text{$k$-times}} \dif z=0.
		\end{gather*}
		Then the following hold.
		\begin{itemize}
			\item Under assumptions (i), (ii), (iii'): given $u\in L^p(\R^d)$ and $\varphi\in W^{k,p'}(\R^d)$, one finds that for each $\eps>0$, there exists $\delta>0$ such that
			\[
			\left|\int_{\R^d}\lp\D_\delta u(x)\rp\varphi(x)\dif x
			-\int_{\R^d}u(x)\lp a\bdot\nabla\varphi(x)\rp\dif x\right|
			\le\eps\delta^{k-1};
			\]
			\item Under assumptions (i), (ii), (iii''): given $u\in L^p(\R^d)$ and $\varphi\in W^{k+1,p'}(\R^d)$, one finds that for each $\eps>0$, there exists $\delta>0$ such that
			\[
			\left|\int_{\R^d}\lp\D_\delta u(x)\rp\varphi(x)\dif x
			-\int_{\R^d}u(x)\lp a\bdot\nabla\varphi(x)\rp\dif x\right|
			\le\delta^k\|\varphi\|_{W^{k+1,p'}(\R^d)};
			\]
			\item Under assumptions (i), (ii), (iii''): given $u\in C^k(\R^d)$, one finds that for each $x\in\R^d$ and $\eps>0$, there exists $\delta>0$ such that
			\[
			|\D_\delta u(x)-a\bdot \nabla u(x)|\le\eps\delta^{k-1}.
			\]
		\end{itemize}
		
	\end{remark}

	\begin{remark}
		The point of Theorem~\ref{thm: convergence to classical derivative} and the previous remark is not to give a new way of approximating derivatives for smooth functions. Rather, the theorem shows that under fairly general conditions on $\mu_\delta$ and $u$, $\D_\delta u$ approximately matches the directional derivative $a\bdot \grad u$. Therefore, we can interpret $\D_\delta$ as behaving roughly like a classical derivative in that it captures something about the rate of change in a direction. This provides an extension of familiar calculus tools to irregular problems: $\D_\delta$ is defined on all of $L^p$, rather than just Sobolev functions, and there is no need for $\Omega_\delta$ to be regular for $\D_\delta$ to be a well-behaved operator. 
		
		This is a crucial point. Rather than working on, say, a fractional Sobolev space, the analysis of this paper is carried out entirely in $L^p$ spaces.
	\end{remark}
	
	Under appropriate hypotheses on $\mu_\delta$, we see that $\D_\delta$ approximates the derivative. Later, we will see it also satisfies integration by parts, meaning it shares some structural properties with the classical derivative as well. However, the following proposition highlights a key difference between the local and nonlocal derivatives: $\D_\delta u$ is often smoother than $u$, while $\grad u$ is typically rougher than $u$. This is related to behavior we will study in Section~\ref{sec: compact} and highlights a key feature of nonlocal modeling: By choosing $\mu_\delta$ appropriately, we can change the behavior of the operator $\D_\delta$ to fit the given setting. If we need the nonlocal derivative to be smooth while still approximating the classical derivative, pick $\mu_\delta$ to be antisymmetric, with zero higher order moments, and smooth enough. If instead we do not want $\D_\delta u$ to be regularizing, pick $\mu_\delta$ such that $\overline{\mu_\delta}\ne 0$.
	
	\begin{Prop}\label{prop: bonus regularity}
		Suppose that $\mu_\delta$ integrates to zero: $\int_{B_\delta(0)} \mu_\delta(z)\dif z = 0$. Fix $p\in[1,\infty]$, and let $u\in L^p(\Omega)$ be given.
		\begin{itemize}
			\item If $\mu_\delta \in L^q(\R^d)$ for $q\in (1,\infty)$, then $\D_\delta u \in L^r(\Omega)$, where $1+1/r = 1/p+1/q$.
			\item If $\mu_\delta \in C^k(\R^d)$ for some $k\in\N$, then $\D_\delta u \in C^k(\Omega)$.
		\end{itemize}
	\end{Prop}
	\begin{proof}
		When $\mu_\delta$ is mean-free, then $\D_\delta$ is essentially a convolution operator. Thus, the first result follows from Young's Inequality and the second follows from a basic fact about convolutions.
	\end{proof}
	\begin{remark}
		The above proposition shows that $\D_\delta$ can sometimes improve integrability or smoothness. More general and refined results of this sort for integral operators have been well-studied. For example, one can specify conditions on $\mu_\delta$ to ensure that $\D_\delta u$ has a certain amount of H\"older regularity~\cite[p. 334]{kantorovich_functional_1982} or merely continuity (See ~\cite{graham_compactness_1979} or Theorem 34.10 of~\cite{serov_fourier_2017}). A more comprehensive list can be found on page 66 of~\cite{precup_methods_2002} and page 39 of~\cite{gripenberg_volterra_1990}. See also~\cite{webb_compactness_2022, lan_compactness_2020, lan_sufficient_2022}, which study integral operators in the context of Caputo and Riemann-Liouville fractional derivatives.
	\end{remark}

	\subsection{Examples}\label{sect: ex}
	We have seen that, for sufficiently small $\delta$ and nice functions $u$, $\D_\delta u$ acts like a directional derivative. In this section we present some concrete examples of how this operator acts on functions, focusing on the case of a one-dimensional domain. Given the discussion of Section~\ref{sect: symm}, it is important to study both antisymmetric and asymmetric choices of $\mu_\delta$.

	The convergence theorem above is stated for a general family of kernels $\set{\mu_\delta}_{\delta>0}$. In practice, typically one fixes a functional form of the kernel and rescales as $\delta$ shrinks. For example, one choice is to pick an odd polynomial $p(z)$ such that \[\int_{-\delta}^\delta \abs{zp(z)} \dif z = 1. \] Then define $\mu_\delta(z)$ to be $\frac1\delta p(z/\delta)$, times an appropriate indicator function. This satisfies the hypotheses of the theorem, so $\D_\delta\approx \od{}{x}$ for small $\delta$. 
	
	Another option is to pick an antisymmetric potential-type function, say \[\rho(z) = \frac1{\abs{z}^{\beta}}\frac{z}{\abs{z}}\chi_{B_\delta(0)}(z),\] where $0<\beta<1$. Then for each $\delta>0$, set $\mu_\delta(z)\defn C_\delta \rho(z)$, where $C_\delta$ is the following scaling constant:
	\[
	C_\delta\defn \left( \int_{-\delta}^\delta z\rho(z) \dif z \right)^{-1}.
	\]
	This ensures that, for any $\delta>0$, $\int_{-\delta}^\delta z\mu_\delta(z)\dif z =1$ and the conditions of the above convergence theorem are met. These potential-type kernels occur naturally in peridynamics~\cite{chen_selecting_2015}, and are similar to the kernels used to define fractional Sobolev-Slobodeckij spaces.
	
	\begin{remark}		
		It is worth noting that the norm of $\mu_\delta$ is unbounded as $\delta\to 0$. This fact is somewhat hidden in the scaling factor $C_\delta$. The scaling requirement in~\eqref{eqt: scaling 1} implies the magnitude of $\int_{B_\delta(0)} z\mu_\delta(z)\dif z$ must approach $\abs{a}$. Hence, $\mu_\delta\sim \delta^{-n-1}$ for $\delta>0$ small and so $\norm{\mu_\delta}_{L^1\left(B_\delta(0)\right)}\to \infty$ as $\delta\to 0$. This can cause difficulties in certain estimates, like when proving convergence of solutions to a local problem \cite[Section 4.5]{safarik_mathematical_2024}.
	\end{remark}	
	
	\begin{ex}\label{ex: muti-D potential}
		We provide a general family of sign-changing kernels with an integrable singularity. Let $a\in\R^d$ and $\beta>1-d$ be given, and suppose that $\theta\in L^\infty(\partial B_1(0))$ satisfies the following: there exists $r>0$ such that
		\[
		\int_{\partial B_1(0)}\theta(\omega)\omega\dif\sigma(\omega)=ra.
		\]
		Here $\sigma$ is the surface measure on the unit sphere $\partial B_1(0)$. For each $\delta>0$, define $\mu_\delta\in L^1(\R^d)$ by
		\[
		\mu_\delta(z)
		=\lp\frac{d+\beta}{r\delta^{d+\beta}}\rp\theta\lp\frac{z}{|z|}\rp|z|^{\beta-1}\chi_{B_\delta(0)}(z).
		\]
		Thus, the family $\{\mu_\delta\}_{\delta>0}$ satisfies (iii'') in Theorem~\ref{thm: convergence to classical derivative}. To verify (i), we compute
		\[
		\int_{\R^d}|\mu_\delta(z)||z|\dif z
		\le\lp\frac{d+\beta}{r\delta^{d+\beta}}\rp\|\theta\|_{L^\infty(B_1(0))}
		\int_{B_\delta(0)}|z|^{\beta}\dif z
		\le C\delta^{-d-\beta}\int_0^\delta s^{d+\beta-1}\dif s
		\le \frac{C}{d+\beta}.
		\]
		For assumption (ii), we find
		\begin{align*}
			\int_{\R^d}\mu_\delta(z)z\dif z
			=&\frac{d+\beta}{r\delta^{d+\beta}}
			\int_0^\delta\int_{\partial B_1(0)}\theta(\omega)\lp s^{\beta-1}\rp\lp s\omega\rp s^{d-1}
			\dif\omega\dif s
			=\frac{d+\beta}{r\delta^{d+\beta}}
			\int_0^\delta s^{d+\beta-1}(ra)\dif s\\
			=&a.
		\end{align*}
		
		For a more specific example, suppose $a=\hat{e}_1\in\partial B_1(0)$, the unit basis vector in the $1$-st coordinate direction. With $0<t<1$, consider the spherical sectors
		\[
		S^+_{t}=\{\omega\in\partial B_1(0): \omega\bdot\hat{e}_1=\omega_1>t\}
		\quad\text{ and }\quad
		S^-_{t}=\{\omega\in\partial B_1(0): \omega\bdot\hat{e}_1=\omega_1<-t\}.
		\]
		Define $\theta=\chi_{S^+_t}$. We find
		\[
		\int_{\partial B_1(0)}\theta(\omega)\omega\dif\sigma(\omega)
		=\int_{\partial B_1(0)}\theta(\omega)\ls\omega_1\hat{e}_1+(\omega-\omega_1\hat{e}_1)\rs\dif\sigma(\omega),
		\]
		Here $(\omega-\omega_1\hat{e}_1)$ is the component of $\omega$ perpendicular to $\hat{e}_1$. If $\omega\in S^+_t$, then so is $\omega_1\hat{e}_1-(\omega-\omega_1\hat{e}_1)$. We conclude that
		\[
		\int_{\partial B_1(0)}\theta(\omega)\omega\dif\sigma(\omega)
		=\lp\int_{S^+_t}\omega_1\dif\sigma(\omega)\rp\hat{e}_1=r\hat{e}_1.
		\]
		The associated $\D_\delta$ provides a nonlocal operator approximating differentiation in the $1$-st coordinate direction. A similar computation shows that the antisymmetric kernel associated with $\theta=\chi_{S^+_t}-\chi_{S^-_t}$ can also be used.
	\end{ex}
	
	The point of the operator $\D_\delta$ is that it ``acts like a derivative" for nice functions while still exhibiting nice operator-theoretic properties for non-smooth functions. The next example is a particular case of Example~\ref{ex: muti-D potential} and illustrates what happens if $\D_\delta$ is applied to a function which is not classically differentiable.

	\begin{ex}\label{ex: potential}   
		Set $\delta = 0.05$ and consider a potential-type kernel of the sort introduced at the start of this subsection:
		\[
		\mu_\delta(z)\defn C_\delta \frac{z}{\abs{z}}\frac{1}{\abs{z}^{2/3}}\chi_{B_\delta(0)}(z).
		\]
		This is a special case of Example~\ref{ex: muti-D potential}, with $\beta = 1/3$ and $\theta$ is defined on just two points (the boundary of the one-dimensional ball). Here, we take $\theta(1)=1$, $\theta(-1)=-1$.
		
		This function is antisymmetric, and $C_\delta$ is a scaling constant (in this case, $C_\delta \approx 36.1922$). Figures~\ref{fig: abs val} and~\ref{fig: cusp} show the classical and nonlocal derivatives of a function with a corner and one with a cusp, respectively. Note that in regions where the function is smooth, the local and nonlocal derivatives agree nearly exactly. But near the point where the function fails to be differentiable, the graph of the nonlocal derivative stays continuous and bounded. In this sense, the nonlocal derivative $\D_\delta$ is nearly a generalization of the classical, local derivative. The plots also reinforce the perspective of interpreting $\D_\delta u$ as a moving weighted average of the slopes of $u$.
		\begin{figure}[!h]
			\centering
			\includegraphics[width=.4\textwidth]{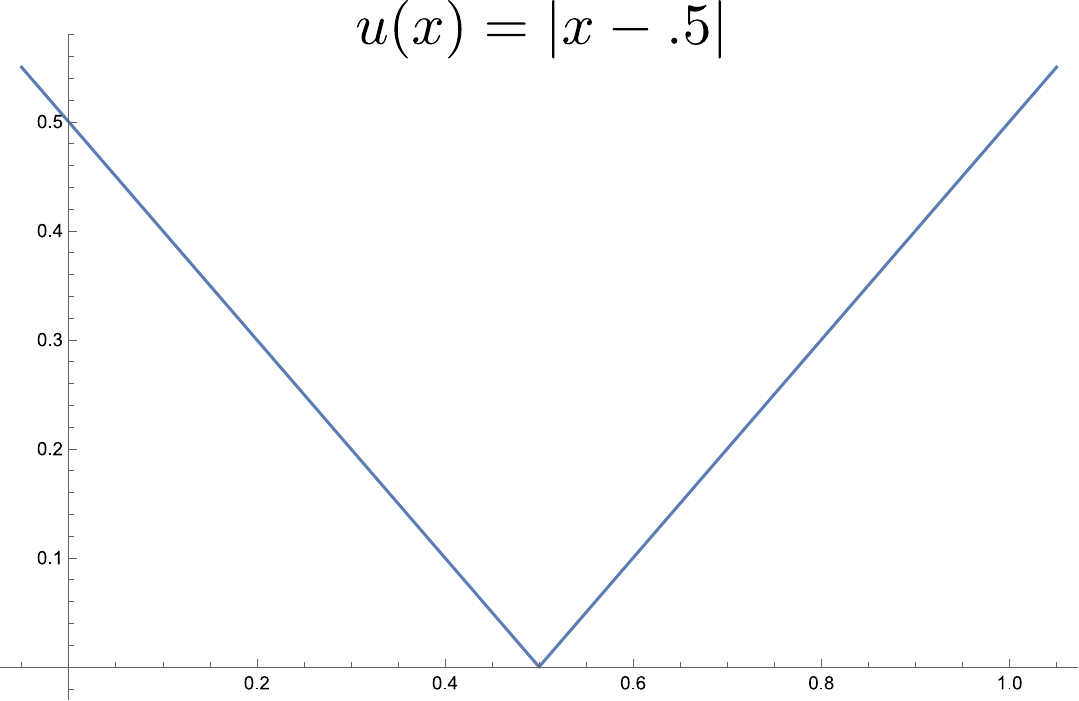}      
			\hspace{15pt}
			\includegraphics[width=.5\textwidth]{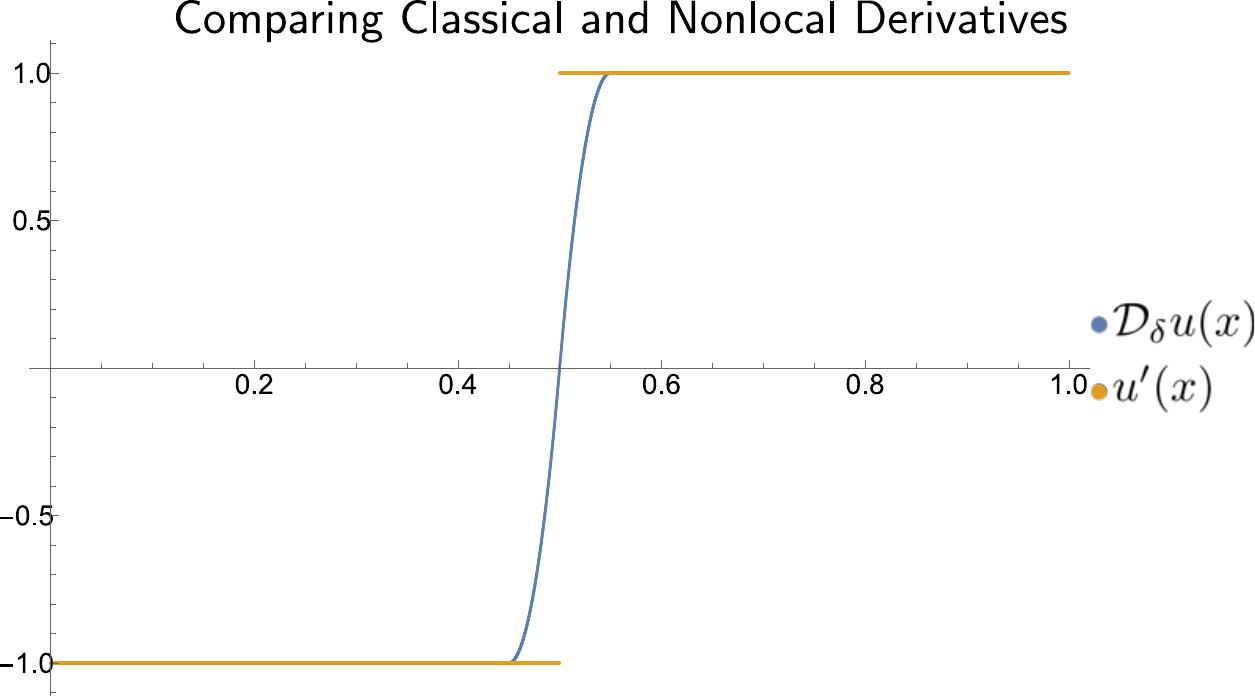}
			\caption{On the left is a plot of the function $u(x)=\abs{x-.5}$, which is not differentiable at $x=0.5$. The right plot shows the local derivative of $u$ in orange and the nonlocal derivative described in Example~\ref{ex: potential} (approximated numerically) in blue.}
			\label{fig: abs val}
		\end{figure}
		\begin{figure}[!h]
			\centering
			\includegraphics[width=.4\textwidth]{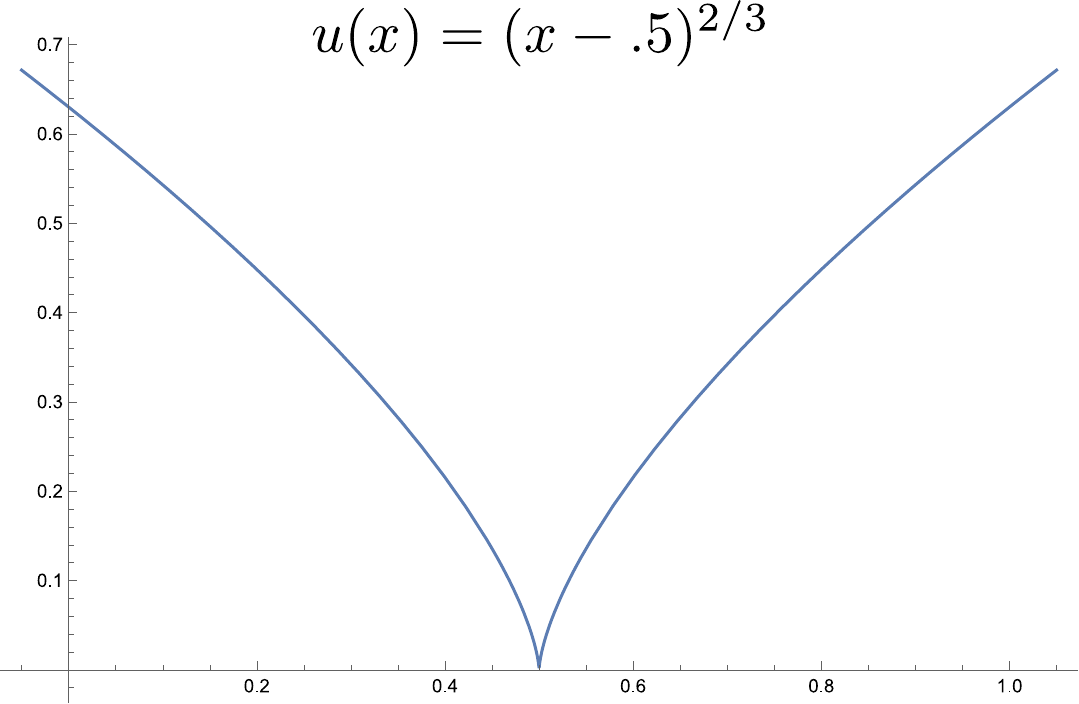}
			\hspace{15pt}
			\includegraphics[width=.5\textwidth]{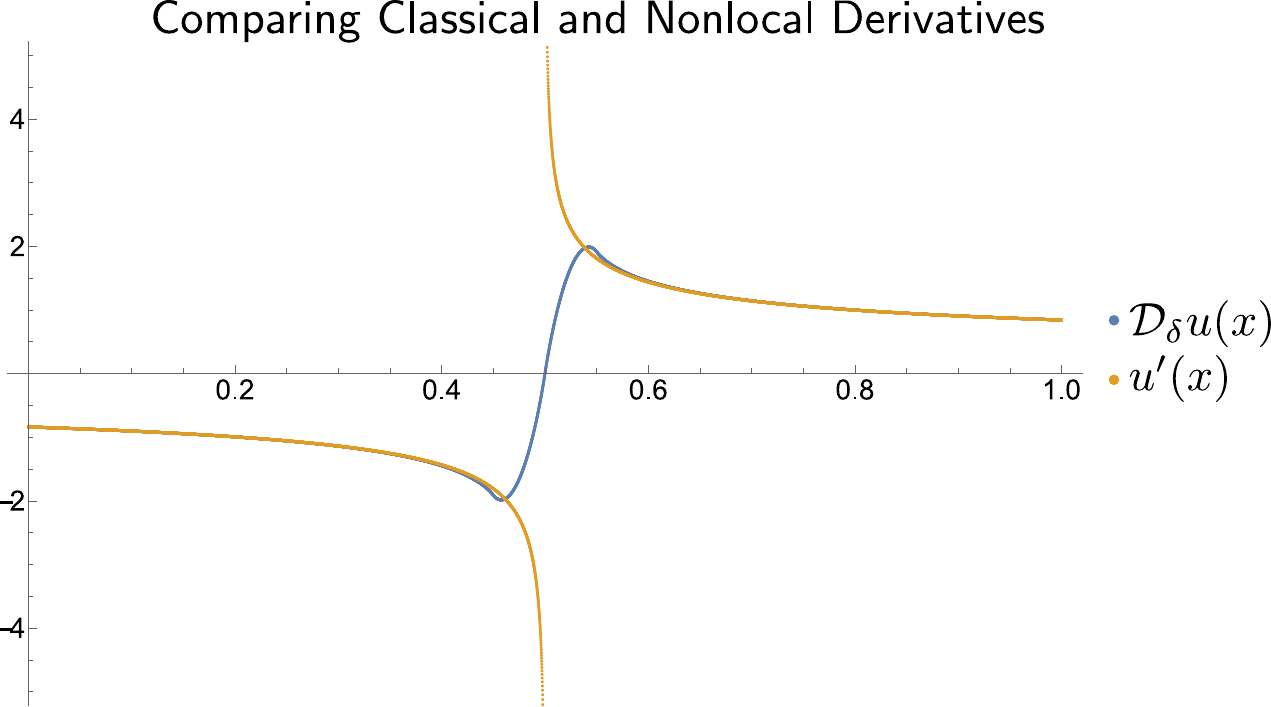}
			\caption{On the left is a plot of the function $u(x)=\abs{x-.5}^{2/3}$, which is not differentiable at $x=0.5$. The right plot shows the local derivative of $u$ in orange and the nonlocal derivative described in Example~\ref{ex: potential} (approximated numerically) in blue.}
			\label{fig: cusp}
		\end{figure}
	\end{ex}
	
	To obtain the pointwise estimate $\abs{\D_\delta u(x) - u'(x)} < C\delta^\alpha$ in part (c) of Theorem~\ref{thm: convergence to classical derivative}, we needed to assume that $u'\in C^{0,\alpha}(\R)$. Next, we give an example where $\abs{\D_\delta u(0) - u'(0)} < C \delta$ despite the discontinuity in $u'$ at zero.
	
	\begin{ex}\label{ex: oscil} 
		Let $\Omega=(-1,1)$ and consider any $\delta\in(0,1)$. Consider the function $u: \Omega_\delta\to\R$
		\[
		u(x) = \begin{dcases}
			x^2\sin\left(\frac{1}{x}\right), & x\ne 0,\\
			0, & x=0.
		\end{dcases}
		\]
		Then $u$ is differentiable on $\dom$, with $u'(0)=0$, and that $u'$ is discontinuous.
		
		Let $\mu_\delta(z)$ denote the piecewise constant, antisymmetric kernel $\mu_\delta(z)= \frac{1}{\delta^2}\mathop{\mathrm{sign}}(z)\chi_{B_\delta(0)}(z)$. This is a particular case of Example~\ref{ex: muti-D potential}, with $\beta = 1$ and taking $\theta:\set{-1,1}\to\set{-1,1}$ to be the identity map.

		After a change of variables, we find
		\begin{align*}
			\D_{\delta} u(0)
			&=\frac 2{\delta^2} \int_{0}^{\delta} y^2\sin\left(\frac 1y\right) \dif y.
		\end{align*}
		This yields the upper bound: $\abs{\D_{\delta} u(0) } \le\frac{2}{3}\delta$.
	\end{ex}
	
	\begin{ex}\label{ex: zero}
		The previous example explored behavior of an antisymmetric kernel as $\delta$ goes to zero. What if $\delta$ is fixed? Continuing with the antisymmetric piecewise constant kernel of Example~\ref{ex: oscil}, consider the function $h(x) = \sin(\frac{2\pi}{\delta} x)$. The nonlocal derivative ``sees" a ball of radius $\delta$ at each point, meaning that $h$ completes 2 full periods within $(x-\delta,x+\delta)$. As the kernel $\mu_\delta$ is antisymmetric and weights every point equally, this means that the contributions from the left and right precisely cancel:
		\[
		\D_\delta h(x) = - \int_{x-\delta}^x h(y) \dif y + \int_x^{x+\delta} h(y)\dif y =0.
		\]
		Thus, the nonlocal derivative of $h$, for the given choice of kernel, is identically zero. Contrast this with the classical derivative, where the only functions with zero derivative are constants. 
		
		This is in part a consequence of working on bounded domains. In~\cite{haar_new_2022}, it is shown that the zero set of the operator $\D_\delta: L^p(\R)\to L^p(\R)$ contains only the zero function. See Section~\ref{sect: symm} and Corollary~\ref{coro: no poincare} for discussion and implications of the fact that the null space of $\D_\delta$ is much larger than the constant functions.
	\end{ex}
	
	Suppose we want $\D_\delta$ to converge to the local derivative. We may initially guess that the symmetric part of the kernel $\mu_\delta$ must get small as $\delta$ shrinks. This next example illustrates that this is not necessarily the case.
	
	\begin{ex}\label{ex: asymm}
			For each $\delta>0$, let $\mu_{\delta}:\R\to\R$ be a one-dimensional special case of Example \ref{ex: muti-D potential}, with $\beta = 1/2$ and $\theta:\set{-1,1}\to\set{0,1}$ given by $\theta(1) = 1$ and $\theta(-1) =0$.
			Note that each function $\mu_\delta$ can be decomposed into symmetric and antisymmetric parts:
			\[
			\mu_{\text{s}}(z)\defn \begin{dcases}
				\frac3{4\delta^{3/2} \abs{z}^{1/2}}, & -\delta<z<\delta\\
				0, & \text{otherwise}
			\end{dcases}, \ \ \ \ 
			\mu_{\text{a}}(z)\defn \begin{dcases}
				\frac{-3}{4\delta^{3/2} \abs{z}^{1/2}}, & -\delta<z<0\\
				\frac3{4\delta^{3/2} \abs{z}^{1/2}}, & 0<z<\delta\\
				0, & \t{otherwise}.
			\end{dcases}
			\]
			
			A straightforward computation shows that, for every positive $\delta$,
			\begin{gather*}
				\int_{-\delta}^\delta \abs{\mu_\delta(z)}\dif z =\frac{3}{\delta},\ \ \ 
				\int_{-\delta}^{\delta} z\mu_\delta(z)\dif z = 1= \int_{-\delta}^{\delta}\abs{z\mu_\delta(z)}\dif z, \  \text{ and } \ 
				\norm{\mu_{\text{s}}}_{L^1\left(\R\right)} =  \frac3\delta.
			\end{gather*}
			Hence, the conditions of the convergence theorem are met. However, the symmetric part does not get small; in fact, its norm grows without bound as $\delta\to0^+$. A careful look at the Taylor series argument, in Section~\ref{sect: motiv}, shows one only needs the second moment of the symmetric part to go to zero:
			$
			\lim_{\delta\to0^+}\int_{-\delta}^{\delta}  z^2 \mu_{\text{s}}(z)\dif z =0,
			$
			which is a weaker condition. In the example above, we find $
			\int_{-\delta}^{\delta}  z^2 \mu_{\text{s}}(z)\dif z = \frac3{5} \delta\to0$ as $\delta\to0^+$.
	\end{ex}

	\begin{ex}\label{ex: poly}
		In the previous example of an asymmetric kernel, the support of $\mu_\delta$ was contained in a half space. This is not needed, as this polynomial example shows.
		
		For each $\delta\in(0,2)$, define
		\[
		\mu_\delta(z)\defn \begin{dcases}
			\textstyle{\frac{225}{128}\frac{1}\delta + \frac{3675}{128}\frac{1}{\delta^3}z-\frac{525}{64}\frac{1}{\delta^3}z^2-\frac{6615}{64}\frac{1}{\delta^5}z^3+\frac{945}{128}\frac{1}{\delta^5}z^4+\frac{10395}{128}\frac{1}{\delta^7}z^5}, & -\delta\le z\le \delta\\
			0, &\text{otherwise.}
		\end{dcases}
		\]
		The coefficients are chosen so that, for all $\delta>0$,
		\[
		\int_{-\delta}^\delta \mu_\delta(z)\dif z = 1 =\int_{-\delta}^\delta z\mu_\delta(z)\dif z, \text{ and } \int_{-\delta}^\delta z^j\mu_\delta(z)\dif z =0, \text{ for }j=2,3,4.
		\]
		Hypothesis (i) of Theorem \ref{thm: convergence to classical derivative} for all $\delta<2$, meaning the theorem ensures convergence to the local derivative. In fact, we expect rapid convergence to the local derivative in light of Remark \ref{rmk: higher order}.
	\end{ex}

	\begin{ex}\label{ex: DIC}
		Although we have focused on the symmetry of $\mu_\delta$, Corollary~\ref{coro: no poincare} shows that another important class of kernels to consider is the mean-free functions. Since we integrate on a symmetric domain, this class includes all antisymmetric kernels. But some applications use kernels which integrate to zero without assuming antisymmetry. See for example~\cite{lehoucq_nonlocal_2014}, which uses a piecewise-linear, asymmetric kernel to define a nonlocal strain in digital image correlation. The resulting nonlocal derivative will converge to the classical derivative for smooth functions, but it will not typically satisfy integration by parts or a nonlocal Poincar\'e inequality.
	\end{ex}

	\subsection{Vector Calculus}\label{sect: vectors}
	
	Sections~\ref{sect: convergence} and~\ref{sect: ex} justify interpreting the operator $\D_\delta$ as a directional derivative. This means we can form linear combinations of these operators, as well as vector-valued versions, analogous to classical vector calculus. 
	
	For each $k=1,\dots, d$, denote the $k$-th standard basis vector in $\R^d$ by $\hat{e}_k\in\R^d$. Given $m\in\mathbb{N}$ and $p\in[1,\infty]$, let  $L^p\left(\Omega_\delta\right)^m$ denote the direct sum of $L^p(\Omega_\delta)$ with itself, $m$ times.
	
	Suppose $\set{\vec\mu_{\delta}}_{\delta>0}=\set{(\mu_{\delta,1},\dots,\mu_{\delta,m})}_{\delta>0}\subseteq L^1(\R^d)^m$ such that $\supp\vec\mu_\delta\subseteq\overline{B_\delta(0)}$. For each $k=1,\dots,m$, let $\overline{\D_{\delta,k}}: L^p\left(\Omega_\delta\right)\to L^p\left(\Omega_\delta\right)$ denote the extended nonlocal derivative with kernel $\mu_{\delta, k}$.
	
	\begin{definition}\label{def: nonlocal differential operators}\label{def: div and grad}
		Let $c_1,\dots,c_m\in\R$ be given. We introduce three new operators:
		\begin{gather*}
			\mathscr A_\delta: L^p(\Omega_\delta)\to L^p(\Omega_\delta), \ \ \mathscr A_\delta u \defn \sum_{k=1}^m c_k\overline{\D_{\delta,k}} u\\
			\mathscr G_\delta: L^p(\Omega_\delta)\to L^p(\Omega_\delta)^m, \ \ \G_\delta u \defn \left(
			c_1 \overline{\D_{\delta,1}} u,
			c_2 \overline{\D_{\delta,2}} u,
			\dots,
			c_m \overline{\D_{\delta,m}} u
			\right)\\
			\mathscr{D}_\delta : L^p(\Omega_\delta)^m \to L^p(\Omega_\delta), \ \ \mathscr D_\delta \vec u = \mathscr D_\delta \left(
			u_1,
			u_2,
			\dots,
			u_m
			\right) = \sum_{k=1}^m c_k \overline{\D_{\delta,k}} u_k.
		\end{gather*}
		
		We refer to $\mathscr A_\delta$ as a first-order, \textit{linear nonlocal differential operator}. If $m=d$, $c_k=1$, and the directional derivatives are taken along coordinate directions for all $k\in\set{1,2,\dots,d}$, then $\mathscr D_\delta$ is can be identified as a \textit{nonlocal divergence} operator and $\G_\delta$ as a \textit{nonlocal gradient}. We use $\mathscr T_\delta$ as a placeholder, to refer to any of these nonlocal operators. 
	\end{definition}
	
	The next proposition says we can view these three operators as nonlocal analogs of familiar classical differential operators. 
	
	\begin{Prop}\label{prop: convergence of vect ops}
		Suppose that $\Omega \subseteq\R^d$ is open and bounded, and $\Omega_\delta$ is an extension domain for each $\delta>0$. Using the assumptions and notation above, further suppose that for each $k=1,2,\dots,d$, the family of functions $\set{\mu_{\delta,k}}_{\delta>0}$ satisfies (i), (ii), and (iii') of Theorem~\ref{thm: convergence to classical derivative}, with $a=\hat{e}_k$.
		
		Then $\mathscr A_\delta$, $\G_\delta$, and $\mathscr D_\delta$ are bounded linear operators on their respective domains. Additionally, if $u\in W^{1,p}\left(\Omega_\delta\right)$ and $Au \defn \sum_{k=1}^m c_k \partial_k u$, then as $\delta\to0^+$,
		\[
		\norm{\mathscr A_\delta u - Au}_{L^p(\dom)} \to 0, \ \ \norm{ \mathscr G_\delta u - \grad u}_{L^p(\dom)^d} \to 0, \ \text{ and } \ \norm{\mathscr D_\delta \vec u - \mathop{\mathrm{div}} \vec u}_{L^p(\dom)}\to 0,
		\]
	\end{Prop}
	
	\begin{proof}
		The fact that $\mathscr A_\delta$, $\G_\delta$, and $\mathscr D_\delta$ are bounded and linear follows immediately from Proposition~\ref{prop: D is bdd}, since they are finite combinations of the operators $\D_{\delta,k}$. The convergence to the classical, local analog is a direct consequence of Theorem~\ref{thm: convergence to classical derivative}.
	\end{proof}
	
	\begin{remark}
		Note that the kernels $\mu_{\delta,k}$ need not have the same functional form for all $k$. More generally, the proposition holds for tensor-combinations of first-order nonlocal derivatives.
	\end{remark}
	
	We next show that the nonlocal gradient arises in the context of peridynamics. This provides a key piece of motivation for studying these operators: the nonlocal gradient is well-defined for any $L^p$ deformation and can therefore model fracture and sharp corners without modification.
	
	\begin{ex}\label{ex: peridynamic correspondence}
		In state-based peridynamics~\cite{silling_peridynamic_2007}, a nonlocal reformulation of continuum mechanics, a key object is the \textit{deformation vector state field} $\underline{\mathbf{Y}}$. For each point $\bx$ and each time $t$, $\underline{\mathbf{Y}}[\bx,t]$ is a \textit{vector state}, or a vector-valued map defined on $\overline{B_\delta(0)}$. As the body is deformed, suppose that the point $\bx$ is sent to the point $\by(\bx,t)$. Then for any other point $\bx'\in \overline{B_\delta(\bx)}$, $\underline{\mathbf{Y}}[\bx,t]$ maps the bond $\bx'-\bx$ to the corresponding deformed bond: $    \underline{\mathbf{Y}}[\bx,t]\langle\bx'-\bx\rangle \defn \by(\bx',t) - \by(\bx,t).
		$
		This map need not be linear or continuous, allowing for more general deformations than the classical framework.
		
		The theory of \textit{correspondence} provides a method of roughly translating between a classical material model and a state-based peridynamic model. Given a tensor, one can \textit{expand} it to a state, and, given a state, one can \textit{reduce} to a tensor (reduction is the inverse of expansion). In this setting, one defines a tensor $\overline{\mathbf F}$ that plays a role similar to the classical deformation gradient. Given a deformation $\underline{\mathbf{Y}}$, the \textit{approximate deformation gradient} $\overline{\mathbf F}$ is the reduction of $\underline{\mathbf{Y}}$. First we write this in the usual notation of peridynamics, where $\xi$ denotes the bond between $\bx'$ and $\bx$, then we write the same expression in the notation of this paper:
		\begin{align*}
			\overline{\mathbf F}(\bx,t) &= \left(\int_{\mathcal H} \underline{\omega}\langle\xi\rangle\underline{\mathbf{Y}}[\bx,t]\langle\xi\rangle \otimes \xi \dif {V_{\xi}}\right) \left(\int_{\mathcal H} \underline{\omega}\langle\xi\rangle\xi\otimes\xi\dif{V_\xi}\right)^{-1} \\
			&=  \left(\int_{B_\delta(0)} [\vec y(x+z,t)-\vec y(x,t)] \otimes\vec{\mu_\delta}(z) \dif {z}\right) \left(\int_{B_\delta(0)}\vec{\mu_\delta}(z)\otimes z \dif z\right)^{-1}.
		\end{align*}
		where $\vec{\mu_\delta}(z)\defn \underline{\omega}\langle z\rangle z$ can be identified as the kernel of the nonlocal operator, and $\underline{\omega}$ is the \textit{influence function}, which determines which bonds contribute to the force acting at a point. It is assumed in peridynamics that $\underline\omega:\overline{B_\delta(0)}\to[0,\infty)$ is strictly positive on a set of positive measure, and to fit into the current framework we further require $\underline{\omega}\langle z\rangle z$ to be integrable. In the last equation above, the first integral can be seen as a nonlocal gradient of the deformation map. Each component of the tensor will be a nonlocal directional derivative of some coordinate function $y_j$, in the direction of a coordinate unit vector $\hat e_k$. The second integral is called the \textit{shape tensor} in peridynamics, and may be complicated in general. But if we choose $\vec \mu_\delta$ such that $\int_{B_\delta(0)} \mu_{\delta,k}(z) z_j \dif z$ converges to the Kronecker delta $\delta_{jk}$ as the nonlocality vanishes, then the second integral converges to the identity as $\delta\to0^+$.
		
		Hence, $\overline{\mathbf F}$ is like a nonlocal version of the classical deformation gradient, recovering the classical picture within the peridynamic framework. If the components $\mu_{\delta,k}$ satisfy the hypotheses of Theorem~\ref{thm: convergence to classical derivative}, then it can be seen as an approximation of the distributional derivative of any $L^p$ deformation.
	\end{ex}
	
	\begin{remark}\label{rmk: not derivs}
		We have seen that, under mild scaling requirements, the operator $\D_\delta$ acts like a derivative. However, the later results of this paper do not assume that these requirements are satisfied. Therefore, we can think of these nonlocal operators as more general integral operators that are capturing different features of the function they act on. This flexibility is part of what makes nonlocal modeling so powerful. Additionally, there is no need to require $m=d$ in the definition of $\G_\delta$: we may want to model more features than there are dimensions if $\D_\delta$ is being used as more than a simple directional derivative.
	\end{remark}
	
	\section{Adjoints and Integration By Parts}\label{sect: adjoint}
	In this section we derive an expression for the (Banach space) adjoints of the operators discussed in Section~\ref{sect: Operator}. We also obtain a nonlocal analog for integration by parts (Theorem \ref{thm: by parts}). From this, we derive a nonlocal analog of both Green's Identity (Remark \ref{rmk: greens}) and the Divergence Theorem (Corollary \ref{prop: div grad adj}). These results further reinforce the idea that $\D_\delta$ is an integral operator that ``acts like a derivative." The characterization of the adjoint and integration by parts will also be useful for the variational problems studied in Part II.
	
	First we introduce some notation. 
	\begin{definition}\label{def: pairing}
		If $p\in[1,\infty)$ and $p'$ is its H\"older conjugate ($\frac1p+\frac1{p'}=1$), we refer to the following bilinear form $\langle\cdot,\cdot\rangle:L^p( \Omega_\delta) \times L^{p'}( \Omega_\delta) \to \R$ as a \textit{duality pairing}:
		$
		\langle u,v\rangle\defn \int_{\Omega_\delta} u(x)v(x)\dif x.
		$
		For $\vec u \in L^p(\Omega_\delta)^m$ and $\vec v\in L^{p'}(\Omega_\delta)^m$, we instead use $\langle \vec u, \vec v \rangle_{L^p(\Omega_\delta)^m} = \sum_{k=1}^m \langle u_k, v_k \rangle_{L^p(\Omega_\delta)}$.
		
		We say that a bounded linear operator $S^*: L^p( \Omega_\delta) \to L^p( \Omega_\delta)$ is the \textit{adjoint} of $S: L^{p'}( \Omega_\delta)\to L^{p'}( \Omega_\delta)$ if
		$
		\langle S^* u, v \rangle = \langle u, S v\rangle,$ for all $u\in L^p\left( \Omega_\delta\right)$, and all $v\in L^{p'}\left( \Omega_\delta\right)$.	
	\end{definition}
	
	A main result of this section is establishing the form of the adjoint for the operator $\overline{\D_\delta}$. As a corollary, we obtain a formula analogous to classical integration by parts. The result is similar to existing work (e.g, Theorem 1.4 of~\cite{mengesha_localization_2015}, Proposition 5 of~\cite{haar_new_2022}, Theorem 2.7 of~\cite{mengesha_characterization_2016} and Theorem 3.2 of~\cite{bellido_non-local_2023}). Here, however, we focus on bounded domains, we do not assume $\mu_\delta$ is antisymmetric, and we allow $p\ne 2$.
	
	\begin{thm}\label{thm: adjoint}
		Suppose $p\in [1,\infty)$. If $\overline{\D_\delta}$ acts on $L^{p'}\left(\Omega_\delta\right)$, define an operator $\overline{\D_\delta}^*$ on $L^p\left(\Omega_\delta\right)$ by
		\begin{equation}\label{eqt: adjoint defn}
			\overline{\D_\delta}^* u(x) =\int_{\Omega_\delta} u(y)\mu_\delta(x-y)\chi_{B_\delta(x)}(y)\chi_{\Omega}(y) \dif y - u(x) \chi_\Omega(x)\int_{B_\delta(0)}\mu_\delta(z)\dif z.	    
		\end{equation}
		Then $\overline{\D_\delta}^*$ is a bounded linear operator which is the adjoint of $\overline{\D_\delta}$ in the sense of Definition~\ref{def: pairing}.
	\end{thm}

	\begin{proof}
		The operator $\overline{\D_\delta}^*$ is bounded and linear for the same reason $\overline{\D_\delta}$ is: it is the sum of a bounded integral operator and a multiplication operator. 
		
		Recall from equation~\eqref{eqt: decompose extended domain} and the subsequent discussion that $\overline{\D_\delta}$ can be written as
		$
		\overline{\D_\delta} = \overline{I_\delta} -M_{g}.
		$
		It follows immediately from the linearity of the duality pairing $\langle\cdot,\cdot\rangle$ that the adjoint of $\overline{\D_\delta}$ must be given by $\overline{I_\delta}^* - M_{g}^*$.
		So it is sufficient to find the adjoints of the integral operator $\overline{I_\delta}$ and the multiplication operator $M_{g}$.
		
		The adjoint of the multiplication operator is itself:
		\[
		\langle u, M_{g} v\rangle = \int_{\Omega_\delta} u(x) M_{g} v(x) \dif x
		= \int_{\Omega_\delta} u(x) g(x) v(x) \dif x
		=\langle M_{g} u, v\rangle.
		\]
		These integrals make sense because $g\in L^\infty$ and $u$ and $v$ live in H\"older conjugate spaces.
		
		It follows from Example VI.1.6 in~\cite{conway_course_2019} that the adjoint of $\overline{I_\delta}$ is another integral operator, with kernel $K^*(x,y)\defn K(y,x)$. (See equation~\eqref{eqt: kernel of int op} for a definition of $K$.) Hence,
		\[
		\overline{I_\delta}^* u(x) = \int_{\Omega_\delta} u(y)\mu_\delta(x-y)\chi_{B_\delta(x)}(y)\chi_{\Omega}(y) \dif y,
		\]
		which exactly matches the expression given in the statement of the theorem. Therefore, we have shown the adjoint of $\overline{\D_\delta}$ is given by $
		\overline{\D_\delta}^* = \overline{I_\delta}^*- M_{g},
		$
		as desired.
	\end{proof}
	
	\begin{remark}\label{rmk: complex}
		In some applications, it is more convenient to work in a Hilbert space over $\C$ rather than $\R$. For this, we modify Definition~\ref{def: pairing} to include a conjugate in the second argument to match the inner product on the complex Hilbert space $L^2$: $
		\langle u, v\rangle = \int_{\Omega_\delta} u(x) v(x)^*\dif x.
		$
		An analogous result to Theorem~\ref{thm: adjoint} then holds for complex-valued functions. The adjoint of $M_{g}$ is now $M_{{g}^*}$, multiplication by the conjugate of $g$, and the adjoint of $\overline{I_\delta}$ is
		\[
		\overline{I_\delta}^* u(x) = \int_{\Omega_\delta} u(y){\mu_\delta(x-y)}^*\chi_{B_\delta(x)}(y)\chi_{\Omega}(y) \dif y. 
		\]
		Hence, the same formula for $\overline{\D_\delta}^*$ applies in the complex case by simply placing a complex conjugate over the function $\mu_\delta$ both times it appears in equation~\eqref{eqt: adjoint defn}. If $\mu_\delta$ is real-valued (even if $u$ and $v$ are complex-valued), then the form of the adjoint is identical to the result just proven.
	\end{remark}
	
	\begin{remark}
		Theorem~\ref{thm: adjoint} only assumes the kernel is $L^1$ with support in $\overline{B_\delta(0)}$. When $\mu_\delta$ is antisymmetric and $\overline{\D_\delta}$ is restricted to $L^2_0$, it follows immediately from the theorem that $\overline{\D_\delta}$ is skew-adjoint. Applying this componentwise, one quickly recovers the familiar relationship between the gradient and divergence. In Proposition \ref{prop: div grad adj}, we provide an alternative proof based on Theorem \ref{thm: by parts} with no symmetry assumptions. Thus, the result holds for nonlocal operators other than first-order nonlocal derivatives. For example, for a symmetric, as in the nonlocal-Laplacian, it immediately follows from Theorem~\ref{thm: adjoint} that the operator is self-adjoint when acting on $L^2_0$.
	\end{remark}
	
	Now we present the nonlocal analog of integration by parts. As in the classical case, the formula provides a connection between integrals over the domain and the boundary. In the nonlocal setting, however, the boundary term is replaced with a collar integral that depends on the values of the function in a neighborhood of the topological boundary of $\Omega$. Analogous, to zero-trace functions, the collar integrals vanish for $u\in L^p_0$ (see Corollary~\ref{coro: by parts no boundary term}).
	
	\begin{thm}\label{thm: by parts}
		Fix $p\in[1,\infty)$ and $\mu_\delta\in L^1(\R^d)$ with support in $\overline{B_\delta(0)}$. Let $\mu_{\text{s}}$ and $\mu_{\text{a}}$ denote the symmetric and antisymmetric parts of $\mu_\delta$. We denote the corresponding extended nonlocal operators as
		\[
		\Ds u(x)\defn \int_{\R^d} [u(y) - u(x)] \mu_{\text{s}}(y-x)\chi_\Omega(x)\dif y \ \ \ \text{ and } \ \ \ \Da u(x)\defn \int_{\R^d} [u(y) - u(x)] \mu_{\text{a}}(y-x)\chi_\Omega(x)\dif y.
		\]
		Note that we have extended the domain of integration to $\R^d$, but the support of $\mu_\delta$ means only points in $\Omega_\delta$ matter. Also, $\Ds+\Da= \overline{\D_\delta}$, the nonlocal operator with kernel $\mu_\delta$.
		
		Then for any $u\in L^p(\Omega_\delta)$ and $v\in L^{p'}(\Omega_\delta)$, we have
		\begin{align*}
			\langle \Ds u, v \rangle - \langle u, \Ds v\rangle &= \int_{\Gamma_{-\delta}}\int_{\Gamma_\delta}  [u(y)v(x)-u(x)v(y)]\mu_{\text{s}}(y-x) \dif y\dif x, \text{ and}\\
			\langle \Da u, v \rangle + \langle u, \Da v\rangle &= \int_{\Gamma_{-\delta}}\int_{\Gamma_\delta} [u(y)v(x) +u(x)v(y)]\mu_{\text{a}}(y-x)\dif y\dif x.
		\end{align*}
		Combining these together, we obtain
		\[
		\langle \overline{\D_\delta} u, v \rangle+\langle u, \overline{\D_\delta} v\rangle =\int_{\Gamma_{-\delta}}\int_{\Gamma_\delta} [u(y)v(x) +u(x)v(y)]\mu_\delta(y-x) -2u(x)v(y)\mu_{\text{s}}(y-x)\dif y\dif x +2\langle u, \Ds v\rangle. 
		\]
	\end{thm}
	
	\begin{proof}
		The argument in the two cases is nearly identical, so we only present details for the symmetric case.
		Using a density argument, it is enough to prove the statement for smooth functions with compact support. Hence, we may freely apply the Fubini-Tonelli Theorem throughout. Expanding the definitions, switching the order of integration, relabeling variables, and using the symmetry of $\mu_{\text{s}}$ yields the following
		\begin{align*}
			\langle \Ds u,v\rangle &= \int_{\Omega_\delta} \int_{\R^d} [u(y)-u(x)]\mu_{\text{s}}(y-x) \chi_\Omega(x) v(x) \dif y\dif x \\
			& = \int_{\Omega_\delta} \int_{\R^d} u(y)\mu_{\text{s}}(y-x) \chi_\Omega(x) v(x) \dif y\dif x - \int_{\Omega_\delta} \int_{\R^d} u(x)\mu_{\text{s}}(y-x) \chi_\Omega(x) v(x) \dif y\dif x \\
			&= \int_{\Omega_\delta} \int_{\R^d} u(x)\mu_{\text{s}}(y-x) \chi_\Omega(y) v(y) \dif y\dif x - \int_{\Omega_\delta} \int_{\R^d} u(x)\mu_{\text{s}}(y-x) \chi_\Omega(x) v(x) \dif y\dif x.
		\end{align*}
		We note the change in the domain of integration on the first term is valid since the integrand is zero whenever $x\notin\Omega_\delta$ and $y\in\Omega$. Next, we subtract and add $v(y)\chi_\Omega(x)$ to obtain
		\begin{align*}
			\langle \Ds u,v\rangle &= \int_{\Omega_\delta} \int_{\R^d} u(x)v(y)[\chi_\Omega(y)-\chi_\Omega(x)]\mu_{\text{s}}(y-x) \dif y\dif x\\
			&\qqqquad+\int_{\Omega_\delta} \int_{\R^d} u(x)[v(y)-v(x)]\mu_{\text{s}}(y-x) \chi_\Omega(x) \dif y\dif x\\
			&= \int_{\Gamma_{-\delta}}\int_{\Gamma_\delta}  [u(y)v(x)-u(x)v(y)]\mu_{\text{s}}(y-x) \dif y\dif x + \langle u, \Ds v\rangle.
		\end{align*}
		In this last step, we used the pointwise values of the map $(x,y)\mapsto [\chi_\Omega(y)-\chi_\Omega(x)]$, as well as the same manipulation as before: switch the order of integration, relabel, and exploit the symmetry of $\mu_{\text{s}}$. The corresponding formula for the antisymmetric part is similarly proved.
	\end{proof}
	
	\begin{remark}\label{rmk: greens}
		When $\mu_\delta$ is antisymmetric, the conclusion of the previous theorem looks more like the familiar integration by parts. The double integral over the double collar plays the role of a boundary term. Alternatively, when $\mu_\delta$ is purely symmetric, we obtain something analogous to Green's Second Identity. Recall that symmetric kernels are often thought of as second-order operators, like the Laplacian. Comparing the nonlocal and local versions of Green's Identity, we see that the terms in parentheses correspond to a normal derivative:
		\begin{align*}
			\int_{\Omega_\delta} \Ds u(x) v(x)  - u(x) \Ds v(x) \dif x &= \int_{\Gamma_{-\delta}}  v(x)\left(\int_{\Gamma_\delta}u(y)\mu_{\text{s}}(y-x) \dif y\right)-u(x)\left(\int_{\Gamma_\delta}v(y)\mu_{\text{s}}(y-x) \dif y\right)\dif x\\
			\int_{\Omega} \triangle  u(x) v(x)  - u(x) \triangle v(x) \dif x &= \int_{\partial \Omega} v(x) \dpd{u}{\nu} - u(x) \dpd{v}{\nu} \dif x 
		\end{align*}
		Finding a nonlocal Green's Identity is of course not a new idea; see for example \cite{gunzburger_nonlocal_2010}.
		
		Further insight can be gained by linearizing $u$ and $v$ about a point in the inner collar, $x\in\Gamma_{-\delta}$. This yields
		\[
		[u(y)v(x) - u(x)v(y)]\mu_{\text{s}}(y-x)\approx v(x) \left(\nabla u\bdot(y-x)\mu_{\text{s}}(y-x)\right) - u(x) \left(\nabla v(x)\bdot(y-x)\mu_{\text{s}}(y-x)\right).
		\]
		After integrating, we get a symmetric weighted average. Since $x$ is in $\Gamma_{-\delta}\subseteq\Omega$ and $y$ ranges over $\Gamma_\delta$, the vector $(y-x)$ always points outward from $\Omega$. Taking the symmetric weighted average of all such vectors should produce something like an outward pointing normal (for a flat boundary, you recover exactly the outward normal). And so, we expect that for smooth boundaries and smooth functions $u, v$,
		\[
		\int_{\Gamma_{-\delta}}  v(x)\left(\int_{\Gamma_\delta}u(y)\mu_{\text{s}}(y-x) \dif y\right)-u(x)\left(\int_{\Gamma_\delta}v(y)\mu_{\text{s}}(y-x) \dif y\right)\dif x \approx\int_{\partial \Omega} v(x) \dpd{u}{\nu} - u(x) \dpd{v}{\nu} \dif x.
		\]
	\end{remark}
	
	The double integral in Theorem \ref{thm: by parts} depends only on the values of both $u$ and $v$ in the double collar $\Gamma_\delta \cup \Gamma_{-\delta}$. This implies the following Corollary.
	\begin{coro}\label{coro: by parts no boundary term}
		Suppose that $\delta>0$, $p$ and $p'$ are give in $(1,\infty)$ with $\frac1p+\frac1{p'}=1$, and either of the following conditions hold:
		\begin{enumerate}
			\item $u\in L^p_0$ and $v\in L^{p'}_0$, or
			\item $u\in L^p\left(\Omega_\delta\right)$ and $v\in L^{p'}_0$ and $v=0$ for a.e. $x$ in $\Gamma_{-\delta}$. 
		\end{enumerate}
		If $\mu_\delta$ is antisymmetric, then $\langle \overline{\D_\delta} u, v\rangle = -\langle u, \overline{\D_\delta}  v\rangle$. If $\mu_\delta$ is instead symmetric, then $\langle \overline{\D_\delta} u, v\rangle = \langle u, \overline{\D_\delta}  v\rangle$.
	\end{coro}
	
	We also highlight a special case of the corollary above, relevant to applications of spectral theory to nonlocal operators. An analogous statement holds for complex-valued functions.
	
	\begin{coro}\label{coro: adj on Hilbert space}
		Assume that $\mu_\delta$ is $\R$-valued, but view the nonlocal derivative as acting on the complex Hilbert space $L^2_0$. When $\mu_\delta$ is antisymmetric $\mu_\delta$, $\overline{\D_\delta}$ is skew-Hermitian with a purely imaginary spectrum. For symmetric $\mu_\delta$, $\overline{\D_\delta}$ is self-adjoint with a purely real spectrum.
	\end{coro}
	
	Next we apply these results to the nonlocal gradient and divergence, defined in Definition~\ref{def: div and grad}. The conclusion is analogous to the classical case. As with the results described above, the case where $\mu_\delta$ is antisymmetric is already understood. See for example Theorem 3.4 of \cite{du_nonlocal_2019}.
	
	\begin{Prop}\label{prop: div grad adj}
		Fix $p\in[1,\infty)$ and assume the hypotheses given in Section~\ref{sect: vectors} and Definition~\ref{def: div and grad}. 
		\begin{enumerate}
			\item If $X = L^p(\Omega_\delta)$, then $
			\langle \G_\delta u, \vec v \rangle_{X^d} = \langle u, -\mathscr D_\delta \vec v\rangle_X  + \mathsf{A}$, where
			\[\mathsf{A} = \int_{\Gamma_{-\delta}}\int_{\Gamma_\delta} [u(y)\vec v(x) +u(x)\vec v(y)]\bdot \vec\mu_{\delta}(y-x) -2u(x)\vec v(y)\bdot \vec \mu_{\text{s}}(y-x)\dif y\dif x+ \sum_{k=1}^d\ 2\langle u, \overline{\D}_{\text{s},k} v_k\rangle_X.\]
			\item Suppose that $\vec{\mu}_{\delta}$ has antisymmetric components, $u\in L^p_0(\dom_\delta)$, and $\vec v = (v_1,\dots,v_d)\in L^p_0(\dom_\delta)^d$. Then, $\langle \G_\delta u, \vec v \rangle_{X^d} = \langle u, -\mathscr D_\delta \vec v\rangle_X$.
			\item Suppose that $\vec{\mu}_{\delta}$ has antisymmetric components. Then we obtain a nonlocal analog of the Divergence Theorem:
			\begin{align*}
				\int_{\Omega_\delta} \mathscr D_\delta \vec v \dif x  = \int_{\Gamma_{-\delta}}\int_{\Gamma_\delta} [\vec v(x) + \vec v(y)]\bdot \vec \mu_{\delta}(y-x) \dif y\dif x
			\end{align*}
		\end{enumerate}
	\end{Prop}
	
	\begin{proof}
		For (1), applying the nonlocal integration by parts to each component of $\vec v$ allows us to write
		\begin{align*}
			\langle \G_\delta u, \vec v \rangle_{X^d} &= \sum_{k=1}^d \langle \overline{\D_{\delta,k}}u, v_k\rangle_X = \sum_{k=1}^d\lp\langle u,-\overline{\D_{\delta,k}} v_k\rangle_X   +2\langle u, \overline{\D}_{\text{s},k} v_k\rangle_X+ \mathsf{A}_k\rp \\
			&= \langle u , -\mathscr D_{\delta} \vec v\rangle_X + \sum_{k=1}^d\lp 2\langle u, \overline{\D}_{\text{s},k} v_k\rangle_X+ \mathsf{CI}_k\rp,
		\end{align*}
		where $\mathsf{CI}_k$ is the collar integral from Theorem \ref{thm: by parts}: 
		\[
		\mathsf{CI}_k = \int_{\Gamma_{-\delta}}\int_{\Gamma_\delta} [u(y)v_k(x) +u(x)v_k(y)]\mu_{\delta,k}(y-x) -2u(x)v_k(y)\mu_{s,k}(y-x)\dif y\dif x.
		\]
		Summing over $k$ yields the claim.
		
		In (2), the collar integral vanishes, as each function $u$ and $v_k$ are identically zero on $\Gamma_\delta$. Since $\vec{\mu}_{\delta}$ has antisymmetric components, the last summation in $\mathsf{A}$ is zero. Finally, assuming $\vec{\mu}_\delta$ is antisymmetric, (3) follows by applying Part 1 with $u(x) = 1$.
	\end{proof}

	\section{Compactness of the Nonlocal Derivative}\label{sec: compact}

	The nonlocal derivative with integrable kernel acts as a bounded linear operator on $L^p$ into $L^p$. In fact, in many cases of practical interest, $\D_\delta$ turns out to be a compact operator. This is a very strong continuity condition, which implies a host of useful analytic properties. In a some sense, compact operators are ``almost finite-dimensional" and many familiar tools from finite-dimensional linear algebra can be applied. In particular, when $\mu_\delta$ is antisymmetric and $\overline{\D_\delta}$ acts on $L^2_0$, we can leverage the powerful theory of compact, normal operators on a Hilbert space.
	
	Compactness will play a central role in our existence of minimizers result, given in Part II. It also has an important connection to the existence of a Poincar\'e-type inequality, which is explored in Section~\ref{sect: poincare}. We now characterize the compactness of $\D_\delta$ in terms of the choice of kernel $\mu_\delta$. Similar results are given in Lemma 4.5 and Theorem 4.6 of~\cite{mengesha_analysis_2013}, and a related idea is discussed in Theorem 3.4 of~\cite{mengesha_bond-based_2014}. More general results on the compactness of integral operators can be found in Theorem 5.2 of~\cite{precup_methods_2002}.
	
	\begin{thm}\label{thm: cpt}
		For any $p\in[1,\infty)$, the operator $\overline{\D_\delta}: L^p\left(\Omega_\delta\right)\to L^p\left(\Omega_\delta\right)$ is compact if and only $\mu_\delta$ integrates to zero.
	\end{thm}
	
	For some context, refer to Section~\ref{sect: symm}. Note that antisymmetric kernels are necessarily mean-free, and thus the corresponding nonlocal derivative will be compact. The above theorem follows almost immediately from the next lemma, which states that cross-correlations with $\mu_\delta$ are always compact operators when acting on bounded domains.
	
	\begin{lemma}\label{lem: cross cor are cpt}
		Recall that $\Omega_\delta\subseteq \R^d$ is bounded. Fix any function $\mu_\delta\in L^1(\R^d)$ with support contained in $\overline{B_\delta(0)}$. Then define a linear operator $S:L^p\left(\Omega_\delta\right)\to L^p\left(\Omega\right)$ by taking the cross-correlation with $\mu_\delta$:
		\[
		Su(x) = (u\star \mu_\delta)(x) = \int_{\Omega_\delta} u(y)\mu_\delta(y-x)\dif y.
		\]
		(To align with the familiar definition, one can extend all functions by zero to all of $\R^d$.) Then $S$ is a compact operator, for any $p$.
	\end{lemma}
	
	\begin{proof}
		Cross-correlations are similar to convolutions, and convolution operators are compact when acting on bounded domains. A proof of this can be found in Corollary 4.28 of~\cite{brezis_functional_2010}, which relies on the Frechet-Kolmogorov-Riesz Theorem on $L^p$-compactness. A similar approach is presented in Theorems 2.4 and 2.5 of~\cite{gripenberg_volterra_1990}, using the Arzela-Ascoli Theorem. Related results are given in Theorem 3 of~\cite{graham_compactness_1979}, or Theorem 4.1 of~\cite{lan_sufficient_2022}.
	\end{proof}
	The following well-known fact states that the set of compact operators forms a two-sided ideal in the space of bounded linear operators (see~\cite{murphy_c-algebras_1990}).
	\begin{lemma}\label{prop: ideal}
		The product of a bounded operator with a compact operator is always compact. Additionally, finite sums of compact operators are again compact.
	\end{lemma}
	
	We now prove Theorem~\ref{thm: cpt}.
	\begin{proof}
		For the forward direction, suppose that $\overline{\D_\delta}$ is compact. Then decompose $\overline{\D_\delta}$, as in equation~\eqref{eqt: decompose extended domain} and the subsequent discussion: $\overline{\D_\delta} = \overline{I_\delta} - M_{g}$.
		
		By Lemma~\ref{lem: cross cor are cpt}, $I_\delta$ is compact as an operator on $L^p\left(\Omega_\delta\right)$. By Lemma~\ref{prop: ideal}, $\overline{\D_\delta} - {EI_\delta}$ is also compact. Since $M_{g} = \overline{\D_\delta} -\overline{I_\delta}$, we see that $M_{g}$ must be compact. However, it is a fact from functional analysis that the only compact multiplication operator is the zero operator. A very general result of this form can be found in~\cite{takagi_compact_1992}. This means $g$ is almost everywhere zero, which is only possible if $\overline{\mu_\delta} = \int_{B_\delta(0)} \mu_\delta(z)\dif z =0$.
		
		Now, for the opposite implication. Suppose that $\mu_\delta$ is mean-free. Then the multiplication operator $M_{g}$ disappears. By Lemma~\ref{lem: cross cor are cpt}, $I_\delta$ is compact. Thus, $\overline{\D_\delta} = \overline{I_\delta} = EI_\delta$ is also compact by Lemma~\ref{prop: ideal}.
	\end{proof}
	
	\begin{remark}\label{rmk: extension is comapct}
		The operator $\D_\delta$ can be written $
		\D_\delta=P\overline{\D_\delta},$
		where $P$ is defined in Remark~\ref{rmk: isometric isomorphism}. By Lemma~\ref{prop: ideal}, this means that whenever $\overline{\D_\delta}$ is compact, so is $\D_\delta$. Alternatively, if $\D_\delta$ is compact, so is $\overline{\D_\delta} = E\D_\delta$. Thus, for any $p\in[1,\infty)$, $\overline{\D_\delta}$ is a compact operator $L^p\left(\Omega_\delta\right)\to L^p\left(\Omega_\delta\right)$ if and only if $\D_\delta:L^p\left(\Omega_\delta\right)\to L^p\left(\Omega\right)$ is compact.
	\end{remark}
	
	\begin{Prop}\label{prop: vect ops cpt}
		Let $X= L^p\left(\Omega_\delta\right)$ for some $p\in[1,\infty)$. Suppose that $\vec{\mu}_{\delta}\in L^1(\R^d)^m$ has support in the ball $\overline{B_\delta(0)}$. Then the following are equivalent:
		\begin{enumerate}
			\item Each component of $\vec{\mu}_\delta$ is mean-free.
			\item For each $k$, the nonlocal directional derivative $\overline{\D_{\delta,k}}$ is a compact operator.
			\item The operator $\G_\delta$ is compact.
		\end{enumerate}
		Additionally, if any of the three conditions above hold, then $\mathscr{A}_\delta: X\to X$ and $\mathscr D_\delta: X^m\to X$ are also compact operators.
	\end{Prop}
	
	\begin{proof}
		The equivalence of (1) and (2) is given by Theorem~\ref{thm: cpt}. It suffices to show that (2) and (3) are equivalent.
		
		We first prove that (2) implies (3). For each $k=1,\dots,d$, let $\iota_k: X\to X^m$ be the inclusion map $x\mapsto (0,\dots, x,\dots,0)$. Then each $\iota_k$ is a bounded linear operator. By Lemma~\ref{prop: ideal}, we see that $\iota_k \overline{\D_{\delta,k}}$ is compact, as it is a composition of a compact operator and a bounded operator.
		We can decompose the nonlocal gradient as ${\displaystyle{
				\mathscr G_\delta = \sum_{k=1}^m \iota_k \overline{\D_{\delta,k}}}}$, which is a finite sum of compact operators. Thus, $\G_\delta$ is compact.
		
		For the converse implication, suppose that $\G_\delta$ is compact. The directional derivative $\overline{\D_{\delta,k}}$ is a composition of $\G_\delta$ with a projection. So by the same reasoning as above, the directional derivatives are each compact.
		
		Finally, we verify $\mathscr A_\delta$ and $\mathscr D_\delta$ are also compact whenever Condition 2 holds.	Suppose that each $\overline{\D_{\delta,k}}: X\to X$ is compact. For each $k$, let $P_k: X^m\to X$ be the projection onto the $k$-th coordinate: $
		P_k (\vec v) = P_k\left( (v_1, v_2,\dots,v_m) \right) = v_k.$ Then, ${\displaystyle{
				\mathscr D_\delta u = \sum_{k=1}^m \overline{\D_{\delta,k}} P_k.}}$
		Each $\D_{\delta,k}$ is compact by assumption, and $P_k$ is a bounded linear map. Thus, we may apply Lemma~\ref{prop: ideal} once again to conclude that $\mathscr D_\delta$ is compact. Finally, Lemma~\ref{prop: ideal} implies that finite linear combinations are also compact. Therefore, $\mathscr A_\delta$ is compact operator from $X$ into itself.
	\end{proof}
	
	\subsection{A Connection to the Poincar\'e Inequality}\label{sect: poincare}
	The classical Poincar\'e inequality plays a crucial role in the analysis of PDEs and the calculus of variations. For example, it is related to the coercivity of the Dirichlet energy and eigenvalue problems. It used frequently in energy estimates to show well-posedness; nonlocal versions have been applied to establish nonlocal-to-local convergence for various boundary value problems. Since we have argued that $\D_\delta$ ``acts like a derivative," it is natural to ask whether $\D_\delta$ satisfies an analogous inequality. In the context of variational problems, this question is also related to compact embedding results needed for $\Gamma$-convergence of nonlocal functionals to their local counterparts~\cite{du__2023, han_compactness_2024}.
	
	Many versions of nonlocal Poincar\'e-type inequalities have been established over the past couple of decades. Regarding two-point gradients, we mention~\cite{ponce_estimate_2004, andreu-vaillo_nonlocal_2010, Du_Gunzberger, aksoylu_variational_2011,mengesha_nonlocal_2014, kassmann_solvability_2019}. We will focus on Poincar\'e-type inequalities for one-point nonlocal operators, such as $\D_\delta$, and we state the precise form of the inequality below.
	\begin{definition}\label{def: poincare}
		For a fixed $\delta>0$, a nonlocal operator $\mathscr T_\delta: L^p\left(\Omega_\delta\right)\to L^p\left(\Omega\right)^m$ satisfies a \textit{Poincar\'e inequality} if there is a constant $C$, depending only on $p$, $\delta$, and the domain $\Omega$, such that
		\[
		\norm{u}_{L^p\left(\Omega\right)} \le C \norm{\mathscr T_\delta u}_{L^p\left(\Omega\right)^m}, \ \ \ \forall u\in L^p_0(\dom_\delta).
		\]
	\end{definition}
	Inequalities of this type can be found in~\cite{du_stability_2018, foss2019nonlocal, lee_nonlocal_2020, bellido_non-local_2023, han_nonlocal_2023}. We are, however, unaware of a full characterization of when a Poincar\'e inequality exists for operators like $\D_\delta$. This condition is called \textit{bounded below} and is a strong form of injectivity. We do not yet have a full understanding of the invertibility of $\D_\delta$, or even what its zero set looks like. Throughout Section~\ref{sect: poincare}, we make some connections with the results of Section~\ref{sec: compact} and the question of whether $\D_\delta$ satisfies a Poincar\'e inequality. In particular, we provide some necessary conditions on $\mu_\delta$, and in Part II we explore some sufficient conditions for a nonlocal Poincar\'e inequality.
	
	In Theorem~\ref{thm: cpt}, we saw that many of the kernels $\mu_\delta$ used in practice lead to compact operators $\D_\delta$. Namely, if $\mu_\delta$ is antisymmetric, then $\D_\delta$ is compact. Intuitively, compact operators collapse large spaces into small ones: nonseparable Hilbert spaces are mapped to separable spaces, bounded sets are mapped to precompact sets, and the range of a compact operator is ``almost finite-dimensional" (any compact operators is the norm-limit of finite-rank operators). We therefore expect compact operators to lack injectivity. In fact, compact operators acting on an infinite-dimensional space are never invertible. We show in Corollary~\ref{coro: no poincare} that this intuition is correct: if $\D_\delta$ is compact, it cannot possibly satisfy a Poincar\'e inequality. This follows from the following more general result.
	
	\begin{Prop}\label{prop: inconsisten triad}
		Suppose $X_0$ is a closed subspace of $L^p\left(\Omega_\delta\right)$ and that $S:L^p\left(\Omega_\delta\right)\to L^p\left(\Omega_\delta\right)$ is bounded linear operator. The three conditions below are mutually incompatible; it is impossible for all three to hold simultaneously.
		\begin{enumerate}
			\item The operator $S: L^p\left(\Omega_\delta\right)\to L^p\left(\Omega_\delta\right)$ is compact.
			\item The operator $S$ is bounded below when restricted to $X_0$. That is, there is a constant $C$ depending only on the domain $\Omega$ and the power $p$ such that, for all $u\in X_0$,
			\[
			\norm{u}\le C \norm{S u}.
			\]
			\item The vector space $Y\defn S\left(X_0 \right)$ is infinite-dimensional.
		\end{enumerate}
	\end{Prop}
	
	\begin{remark}
		This is a standard functional analysis fact. Related statements can be found as Exercises 5, 9, and 17 in Chapter 4 of~\cite{bressan_lecture_2012}, Theorems 10.74-10.77 of~\cite{axler_measure_2020}, and Proposition X.1.15 of~\cite{conway_course_2019}.
	\end{remark}

	Applying this Proposition to our setting yields an important corollary. It states that whenever $\mu_\delta$ is antisymmetric and integrable, $\D_\delta$ cannot satisfy a Poincar\'e inequality as in Definition~\ref{def: poincare}. This corollary generalizes the analysis of ``instabilities" and ``zero-energy modes" that arise from integrable, antisymmetric kernels. In fact, our Corollary~\ref{coro: no poincare} can be seen as an extension of Remark 3 of~\cite{huang_stability_2022} and Section 5 of~\cite{du_stability_2018}. These works use Fourier analysis and the Riemann-Lebesgue Lemma to argue that, for $p=2$ on a one-dimensional periodic domain, the nonlocal derivative with antisymmetriic integrable kernel cannot be bounded below.

	\begin{coro}\label{coro: no poincare}
		Let $\mathscr T_\delta$ denote any nonlocal operator defined in Definition~\ref{def: div and grad} such that each component of $\vec{\mu}_{\delta}$ is mean-free. Then $\mathscr T_\delta$ is not bounded below on any closed infinite-dimensional subspace of its domain. In particular, if $\mu_\delta$ is mean-free, then the corresponding operator $\D_\delta$ does not satisfy the Poincar\'e inequality on the space $L^p_0(\dom_\delta)$.
	\end{coro}
	
	\begin{proof}
		Suppose that $X_0$ is a closed subspace of the domain of $\mathscr T_\delta$ such that, for all $u\in X_0$, we have
		\begin{equation}\label{E: Poincare}
			\norm{u} \le C \norm{\mathscr T_\delta u},
		\end{equation}
		where $C$ is independent of $u$. Then $\mathscr T_\delta$ satisfies the first two conditions of Proposition~\ref{prop: inconsisten triad}, meaning the image $\mathscr T_\delta(X_0)$ must be finite-dimensional. The operator $\mathscr T_\delta |_{X_0}: X_0\to \mathscr T_\delta(X_0)$, however, is a bijection: it is sujective by definition and injective by the inequality~\eqref{E: Poincare}. As $\mathscr T_\delta|_{X_0}: X_0\to \mathscr T_\delta(X_0)$ is a linear bijection into a finite-dimensional space, the domain $X_0$ must also be finite-dimensional. Thus, any space where the inequality~\eqref{E: Poincare} holds is finite-dimensional.
		
		Now, for the second part of the corollary. Thanks to Theorem~\ref{thm: cpt}, if $\mu_\delta$ is mean-free then it must be compact. As $L^p_0(\dom_\delta)$ is an infinite-dimensional closed subspace of $L^p\left(\Omega_\delta\right)$, the first part of this corollary implies that there is no valid Poincar\'e inequality on this space.
	\end{proof}

	\section{Comments on Symmetry Assumptions}\label{sect: symm}
	
	Much of the recent literature on convolution-type integral operators assume the kernel is either purely antisymmetric or purely symmetric. With a closer look at the Taylor series argument in Section~\ref{sect: motiv}, we see that the antisymmetric kernels isolate the odd-order derivatives of a smooth function, while symmetric kernels isolate the even-order derivatives. Antisymmetric kernels lead to skew-symmetric integral operators, while symmetric kernels produce self-adjoint operators. This work contributes to the study of nonlocal operators without symmetry assumptions, revealing the interplay between the symmetric and antisymmetric components of the kernel.
	
	We focus on families of integral operators $\{\D_\delta\}_{\delta>0}$, with integrable kernels, that recover the properties of a classical directional derivative as $\delta\to0$. In Theorem~\ref{thm: convergence to classical derivative}, we establish this asymptotic behavior for general asymmetric kernels. This nonlocal-to-local convergence, under the assumption of antisymmetry, has been extensively studied in many works~\cite{Du_Gunzberger,mengesha_localization_2015,shankar_nonlocal_2016, mengesha_characterization_2016,du_robust_2016, shieh_new_2018,delia_helmholtz-hodge_2020, delia_towards_2021, haar_new_2022, delia_connections_2022, cueto_variational_2023}. In Theorem~\ref{thm: convergence to classical derivative}, we see that the primary requirements are $L^1$-boundededness and the convergence of first-moments.

	While $\{\D_\delta\}_{\delta>0}$ converges, in some sense, to a first-order differential operator, the nonlocal derivative has some distinguishing properties. For example, unlike a differential operator, $\D_\delta$ is a continuous operator from $L^p$ into $L^p$. If the kernel is purely antisymmetric, or even just mean-free, the distinctions are more pronounced due to its convolutional structure. Proposition~\ref{prop: bonus regularity}, shows that $\D_\delta u$ directly inherits the differentiability of the kernel, and therefore maps into a proper subspace of $L^p$. Moreover, from Theorem~\ref{thm: cpt}, we see that the operator $\D_\delta$ is not just continuous but actually compact into $L^p$. This severely limits the subspaces of $L^p(\dom_\delta)$ on which a Poincar\'e-type inequality can hold. As defined in Definition~\ref{def: poincare}, a Poincar\'e-type inequality is equivalent to the operator having a bounded left-inverse on $L^p_0(\dom_\delta)$. Proposition~\ref{prop: inconsisten triad} implies nonlocal derivatives with antisymmetric kernels lack a left-inverse except when restricted to finite-dimensional subspace of $L^p(\Omega_\delta)$. In particular, $\D_\delta$ cannot be bounded below on $L^p_0(\dom_\delta)$ or the space of mean-free function, as these are each infinite-dimensional (see Corollary~\ref{coro: no poincare}). Similar observations have been made in~\cite{du_stability_2018, lee_nonlocal_2020, han_nonlocal_2023}. As noted in Section 4.1, the results of this paper generalize ideas from~\cite{huang_stability_2022} and~\cite{du_stability_2018}.

	It has been observed that models based on antisymmetric kernels can lead to issues in numerical simulations. In the context of nonlocal convection problems, standard Galerkin methods and central difference schemes are unstable and lead to ``unphysical oscillations"~\cite{tian_nonlocal_2015, leng_petrov-galerkin_2022}. One solution is to only integrate over a portion of the ball $\overline{B_\delta(0)}$, thereby breaking the symmetry of the kernel. This is analogous to upwinding schemes employed in local problems, except the upwinding is built into the equation itself through the nonlocal operator. Additionally, this symmetry-breaking is related to ensuring that the model is globally mass-conserving and satisfies a maximum principle. For additional discussion, see also~\cite{tian_conservative_2017, du_nonlocal_2017, yu_asymptotically_2022, yu_numerical_2022, safarik_mathematical_2024}. A related issue arises in numerical solutions for state-based peridynamics, where again nonphysical oscillations can develop in simulations. It was discovered that this was a feature of the model itself since the approximate deformation gradient $\overline{\mathbf F}$ has admissible, nontrivial deformations in its null space. The potential energy is inherently insensitive to these zero-energy modes, resulting in instabilities in the model and simulations~\cite{silling_stability_2017} (see Example~\ref{ex: peridynamic correspondence}).
	
	We also briefly mention the importance of symmetry assumptions in compact-embedding type results. These play a key role in the analysis of nonlocal models, and there have been a wide variety of pre-compactness results in the literature; see, for example, Theorem 5.1 of~\cite{mengesha_nonlocal_2014}, Proposition 4.2 of~\cite{mengesha_variational_2015}, Theorem 6.11 of~\cite{andreu-vaillo_nonlocal_2010}, Theorem 4 of~\cite{bourgain_another_2001}, Proposition 4.2 of~\cite{ponce_estimate_2004}, Lemmas 2.2 and 2.3 of~\cite{mengesha_analysis_2013}. In recent work~\cite{du__2023, han_compactness_2024}, these results were extended to the case where the kernel $\rho$ is not necessarily radially symmetric.
	
	On the other hand, there are a number of contexts where the symmetry of the kernel is broken by multiplying by a characteristic function on a half space.~\cite{du_analysis_2017} considers a nonlocal-in-time derivative where the kernel is supported in a half-space. Similar one-sided nonlocal-in-time problems are discussed throughout~\cite{gripenberg_volterra_1990}, as well as Appendix C of~\cite{gal_fractional--time_2020} and Section 3.3 of~\cite{precup_methods_2002}. One-sided kernels can occur in spatial nonlocal derivatives as well. Examples include the traffic models in~\cite{huang_mathematical_2022, du_space-time_2023, huang_incorporating_2023} as well as nonlocal models in the context of Smoothed Particle Hydrodynamics~\cite{lee_asymptotically_2019, du_mathematics_2020}. A more in-depth discussion of nonlocal gradients with kernels supported on a half-ball can be found in~\cite{lee_nonlocal_2020, han_nonlocal_2023}.

	\section{Conclusion}
	This paper lays out some preliminaries needed for the study of variational problems involving first-order nonlocal derivatives without symmetry assumptions. We have generalized existing results on the convergence of nonlocal derivatives to their classical counterparts (Theorem~\ref{thm: convergence to classical derivative}), including connections to peridynamics (Example~\ref{ex: peridynamic correspondence} and Section \ref{sect: symm}). We have also provided a characterization for the adjoint of the nonlocal derivative (Theorem~\ref{thm: adjoint}) as well as conditions under which it is a compact operator (Theorem~\ref{thm: cpt} and Proposition~\ref{prop: vect ops cpt}). This compactness precludes invertibility and interesting Poincar\'{e} inequalitiies (Corollary~\ref{coro: no poincare}). Moreover, the large null space associated with compact operators underlies certain difficulties in nonlocal modeling (Section~\ref{sect: symm}) and numerical instabilities. Applications to existence and regularity of solutions for variational problems will be taken up in Part II of this work.
	
	\section{Acknowledgements}
	We would like to thank Javier Cueto Garcia for many helpful discussions about the nonlocal framework, and Christopher Schafhauser, Mikkel Munkholm, and David Pitts for insight into operator theoretic questions.


\end{document}